\documentclass{amsart}
\usepackage{amsmath} 
\usepackage{stmaryrd}
\usepackage{tikz}
\usetikzlibrary{decorations.pathreplacing}
\usepackage{amsthm,amssymb}
\usepackage[unicode]{hyperref}
\hypersetup{
	colorlinks=true,
	linkcolor=black,
	citecolor=black,
	urlcolor=blue
}
\DeclareMathOperator{\unif}{\mathsf{{R}-Unif}}
\usepackage{xcolor}
\usepackage{mathtools}
\usepackage{calc}
\usepackage{graphicx}
\usepackage{bm}
\usepackage{pifont}
\usepackage{enumitem}

\newtheorem{thm}{Theorem}[section]
\newtheorem{lem}[thm]{Lemma}
\newtheorem{qst}[thm]{Question}
\newtheorem{cor}[thm]{Corollary}
\newtheorem{postulate}[thm]{Postulate}
\newtheorem{prop}[thm]{Proposition}
\newtheorem{pblm}[thm]{Problem}
\newtheorem{conj}[thm]{Conjecture}
\newtheorem{claim}[thm]{Claim}

\theoremstyle{definition}
\newtheorem{defn}[thm]{Definition}
\newtheorem{fact}[thm]{Fact}

\numberwithin{equation}{section}

\newcommand{\lb}{\left(}
\newcommand{\rb}{\right)}
\newcommand{\la}{\langle}
\newcommand{\ra}{\rangle}
\newcommand{\llb}{\llbracket}
\newcommand{\rrb}{\rrbracket}
\newcommand{\rtar}{\rightarrow}

\newcommand{\bbr}{{\mathbb R}}
\newcommand{\bbq}{{\mathbb Q}}

\newcommand{\uphar}{\upharpoonright}
\newcommand{\idl}[1]{{\mathcal #1}}

\newcommand{\bbp}{{\mathbb P}}
\newcommand{\bbn}{{\mathbb N}}

\newcommand{\bfpi}{{\boldsymbol \Pi}}
\newcommand{\bfsgm}{{\boldsymbol \Sigma}}

\newcommand{\infsub}[1]{{[#1]^\omega}}
\newcommand{\finsub}[1]{{[#1]^{<\omega}}}
\newcommand{\inffunc}[1]{{#1^{\omega}}}
\newcommand{\finfunc}[1]{{#1^{<\omega}}}

\newcommand{\bairesp}{{\omega^\omega}}

\newcommand{\prst}{{\mathcal P}}
\newcommand{\suchthat}{{\  \mathrm{s.t.} \  }}

\newcommand{\dom}{{\rm dom}}

\newcommand{\zfcaxiom}{{\sf ZFC}}
\newcommand{\zfaxiom}{{\sf ZF}}
\newcommand{\acaxiom}{{\sf AC}}
\newcommand{\ccaxiom}{{\sf CC}}
\newcommand{\dcaxiom}{{\sf DC}}
\newcommand{\adaxiom}{{\sf AD}}
\newcommand{\chaxiom}{{\sf CH}}

\newcommand{\runifaxiom}{{\sf R\text{-}Unif}}
\newcommand{\pdaxiom}{{\sf PD}}

\newcommand{\choice}{{\sf CP}}

\newcommand{\encfunc}{{\mathrm{e}}}

\newcommand{\finalpha}[1]{\mathrm{Fin}^{#1}}
\newcommand{\terminalsequence}{\operatorname{TS}}
\newcommand{\indexsequence}{\operatorname{In}}
\newcommand{\concatenate}{{}^{\frown}}
\newcommand{\evendiff}{\mathcal{ED}}

\newenvironment{claimproof}[1]{
	\renewcommand{\qedsymbol}{{\tiny Claim #1} \(\square\)}
	\begin{proof}[Proof of Claim #1]
	}{
	\end{proof}
}

\title[Ramsey Property and Pathological Sets]{Ramsey Property and Pathological Sets: Almost Disjointness, Independence and Other Maximal Objects}

\author{Jialiang He}
\address{College of Mathematics, Sichuan University, Chengdu, Sichuan, 610064 China;}
\email{jialianghe@scu.edu.cn}

\author{Jintao Luo$ ^\ast $}
\address{College of Mathematics, Sichuan University, Chengdu, Sichuan, 610064 China;}
\email{jintaoluo@foxmail.com}
\author{Shuguo Zhang}

\address{College of Mathematics, Sichuan University, Chengdu, Sichuan, 610064 China;}
\email{zhangsg@scu.edu.cn}

\subjclass[2020]{Primary 03E15; Secondary 03E25, 03E60}
\keywords{Axiom of Choice; Axiom of Determinacy; Baire property; Lebesgue measurability; Ramsey property; almost disjoint family; independent family; Hamel basis; Vitali set.}

\thanks{$^{\ast}$Corresponding author}

\date{}

\begin{document}
	\begin{abstract}
		We show that under $\mathsf{ZF} + \mathsf{CC}_{\mathbb R}$, if the Ramsey property holds for all sets in a good pointclass $\Gamma$, then there is no MAD family in $\Gamma$, proving a long-standing conjecture made by A.R.D.\ Mathias in 1977. This also holds for $\mathcal I$-MAD families with respect to analytic ideals $\mathcal I$ including $\mathcal{ED}$, $\mathcal{ED}_{\mathrm{fin}}$, and $\finalpha{\alpha}$ for all countable ordinals $\alpha$. Under the same assumption, we show that if any one of the Baire property, Lebesgue measurability or Ramsey property holds for all sets in $\Gamma$, then there is no maximal independent family in $\Gamma$. Under the stronger assumption $\mathsf{ZF} + \mathsf{DC}_{\mathbb R}$, we further prove that if the Ramsey property holds for all sets in $\Gamma$, then $\Gamma$ contains no Vitali sets and thus no Hamel bases.
	\end{abstract}
	
	\maketitle
	
	\section{Introduction}

	\subsection{Background}
	
	The study of regularity properties constitutes a central and enduring theme in the foundations of mathematics. Since the emergence of modern real analysis, mathematicians have sought to delineate the boundary between well-behaved objects, i.e. those that can be approximated by simple geometric or combinatorial configurations, and pathological objects, whose existence challenges our intuitive understanding of the mathematical universe. This dichotomy between regularity and pathology remains a cornerstone of descriptive set theory and continues to drive the field's development.
	
	The classical regularity properties are well-established. A set of real numbers is said to be Lebesgue measurable if it can be assigned a meaningful length or volume. A set has the Baire property if it differs from an open set by a meager set, a topological notion of regularity. A subset of $[\mathbb N]^\omega$, the collection of infinite subsets of $ \bbn $, is said to have the Ramsey property if it admits the existence of infinite homogeneous subsets under partitions. These properties constitute the primary framework for evaluating structural regularity, each capturing a different facet of what it means for a set to be tame.
	
	However, the \emph{Axiom of Choice} $ \acaxiom $, a foundational postulate of modern set theory, guarantees the existence of sets that fail to satisfy certain regularity properties. The Vitali set is the paradigmatic example of a non-measurable set whose existence relies crucially on the Axiom of Choice. Hamel bases, i.e. bases of $\mathbb{R}$ viewed as a vector space over $\mathbb{Q}$, provide another example of pathological phenomena, as they are also non-measurable and their existence implies discontinuous additive functions. In the combinatorial realm, maximal almost disjoint (MAD) families are a central example of such pathology: a MAD family is a maximal collection of infinite subsets of $\mathbb{N}$ with the property that any two distinct sets in the family intersect only finitely. Their existence, typically proved using Zorn's Lemma, yields highly complex combinatorial objects that resist simple description. In addition to MAD families, the literature has also focused on other classical combinatorial pathologies, such as \textbf{m}aximal \textbf{i}n\textbf{d}ependent (MID) families \cite{miller1989infinite}. More recently, $\mathcal I$-MAD families, as generalizations of MAD families with respect to different ideals, have been introduced to probe ever deeper into the structure of pathological combinatorial objects \cite{bakke2022maximal}.
	
	The tension between combinatorial pathologies and structural regularity is particularly evident in the study of Ramsey property. Let $\infsub{A}$ be the set of all infinite subsets of $ A $ where $A\subseteq \bbn$. A set $X\subseteq \infsub{\bbn}$ is called Ramsey if $\exists B\in\infsub{\bbn}$, such that either $\infsub{B}\subseteq X$ or $X\cap \infsub{B} = \emptyset$. One can view a set $X\subseteq \infsub{\bbn}$ as a $2$-coloring of the complete hypergraph $\infsub{\bbn}$, and the Ramsey property asserts that this coloring has an infinite monochromatic hyperclique. In this notation, all subsets of $ \infsub{\bbn} $ being Ramsey is equivalent to the generalized Ramsey theorem $\omega\rtar(\omega)^{\omega}_{2}$.
	
	It turns out that the existence of MAD families is closely related to Ramsey property. The systematic investigation of this topic was initiated by A. R. D. Mathias. In his seminal paper \cite{mathias1977happy}, Mathias introduced the concept of a happy family (now often referred to as a selective coideal) and developed the notion of Mathias forcing. Within this framework, he proved that no analytic MAD family exists \cite[~Corollary 4.7]{mathias1977happy}, and established the consistency of the non-existence of MAD families from a Mahlo cardinal. As he had already shown that all sets are Ramsey in Solovay's model \cite{mathias1969generalization}, a model collapsed from an inaccessible cardinal \cite{solovay1970model}, these results suggested the following questions:
	\begin{conj}{\label{conjecture:con_ramsey}}
		\phantom{a}
		\begin{enumerate}[label=(\arabic*), font=\upshape]
			\item In Solovay's model where all sets are Lebesgue measurable, have Baire property and Ramsey property, are there infinite MAD families?
			\item \cite[~Page 87]{mathias1977happy} Is it a theorem of $\zfaxiom+\dcaxiom_\bbr$ that there are no (infinite) MAD families assuming all sets are Ramsey?
		\end{enumerate}
	\end{conj}

	Towards Conjecture \ref{conjecture:con_ramsey}(1), Asger Dag T\"ornquist proved that there are indeed no infinite MAD families in Solovay's model \cite{tornquist2018definability}. Subsequently, Itay Neeman and Zach Norwood gave another proof of this result and additionally showed that under $\mathsf{AD}^+$ or within $L(\mathbb{R})$ in the presence of large cardinals, MAD families are impossible \cite{neeman2018happy}. Karen Bakke Haga, David Schrittesser and T\"ornquist generalized these results and proved that under $\mathsf{AD}$ and $ V=L(\mathbb{R}) $, or under $\mathsf{AD}^+$ and dependent choice, there are no $\finalpha{\alpha}$-MAD families for all countable ordinals $\alpha$ \cite{bakke2022maximal}. Towards Conjecture \ref{conjecture:con_ramsey}(2), in their celebrated paper \cite{schrittesser2019ramsey}, Schrittesser--T\"ornquist showed that, under the hypothesis that all sets have the Ramsey property augmented by the additional principle $\unif$, there are indeed no MAD families. Furthermore, Schrittesser-T\"ornquist proved that, assuming $ \runifaxiom $, all sets being Ramsey implies no MAD families with respect to all transfinitely iterated Fr\'echet ideals \cite{schrittesser2020ramsey}. Recently, the results on determinacy and MAD families have been generalized to uncountable cardinals by William Chan, Stephen Jackson and Nam Trang \cite{chan2024almost}.
	
	The problem of the consistency strength of MAD families was eventually resolved by Horowitz--Shelah, who showed that the non-existence of MAD families is equiconsistent with $ \mathsf{ZFC} $ \cite{horowitz2019non}. However, whether the extra assumption $\unif$ in the theorem of Schrittesser--Törnquist is necessary remained open.
	
	Among these combinatorial pathologies, the structure and existence of maximal independent families have also been extensively investigated. It is a classical result that there are no analytic infinite MID families, proved by Arnold W. Miller \cite{miller1989infinite} via forcing methods. Projective level MID families are investigated in \cite{brendle2019definable}, and it is shown under $ \zfcaxiom $ that if all sets in $\Gamma$ have the Baire property then there are no infinite MID families in $\Gamma$, where $\Gamma$ is a pointclass closed under existential quantification over reals. Recent investigation has also focused on the spectrum of MID families and its large cardinal version \cite{fischer2019spectrum,eskew2023strong}.
	
	Besides these maximal combinatorial objects, continuous pathologies have also been widely studied. It is well known that Vitali sets are not Lebesgue measurable and have no Baire property, so there are no analytic Vitali sets. While a Hamel basis can be Lebesgue measure zero \cite{sierpinski1920question}, it is also a classical result that there are no analytic Hamel bases, from a measure-theoretic argument \cite{jones1942measure}. This was also proved in \cite{miller1989infinite} via Cohen forcing. And globally, classical regularities indeed destroy Hamel bases classwise. It is folklore that the existence of a Hamel basis implies the existence of a Vitali set \cite{beriashvili2018hamel}, thus if all sets have Baire property or Lebesgue measurability, there are no Hamel bases.

	\subsection{Main Results}
	
	Our first main result establishes that Ramsey regularity suffices to eliminate a fundamental class of combinatorial pathology. Recall that a boldface pointclass is a collection of subsets of Polish spaces closed under continuous preimages and containing all clopen sets \cite[p.~118]{kechris2006axiom}. For our purposes, we call a pointclass \emph{good} if it contains all clopen sets and is closed under countable intersections, finite products, countable unions, continuous images, and continuous preimages. With this terminology in place, we have the following theorem under the \emph{Axiom of Countable Choice for Reals} $ \mathsf{CC}_{\mathbb R} $.
	
	\begin{thm}[$\mathsf{ZF} + \mathsf{CC}_{\mathbb R}$]\label{theorem:introduction}
		Let $\Gamma$ be a good pointclass such that every set in $\Gamma$ has the Ramsey property. Then there is no infinite MAD family in $\Gamma$.
	\end{thm}
	
	This theorem provides a definitive answer to Mathias' conjecture.
	
	Our second main result extends this investigation from combinatorial pathology to the classical non-combinatorial examples. Under the \emph{Axiom of Dependent Choice for Reals} $ \mathsf{DC}_{\mathbb R} $, we prove:
	
	\begin{thm}[$\mathsf{ZF} + \mathsf{DC}_{\mathbb R}$]
		Let $\Gamma$ be a good pointclass such that every set in $\Gamma$ has the Ramsey property. Then $\Gamma$ contains no Vitali sets, and hence no Hamel bases.
	\end{thm}
	
	This result is somewhat unexpected, as Vitali sets and Hamel bases are typically connected to the failure of regularities in the sense of measure and category on the one-dimensional real line, while the Ramsey property serves as a form of regularity in the hypergraph $ [\mathbb{N}]^\omega $, which is a zero-dimensional space. It demonstrates that the Ramsey property, despite its combinatorial nature, can also have implications for the structure of the real line itself, beyond the zero-dimensional combinatorial objects.
	
	Our third main result concerns MID families. An independent family is a collection of subsets of $[\mathbb N]^\omega$ such that any finite Boolean combination of distinct sets from the family is infinite \cite{miller1989infinite}. A MID family is a maximal such family. We prove that:
	
	\begin{thm}[$\mathsf{ZF} + \mathsf{CC}_{\mathbb R}$]
		Let $\Gamma$ be a good pointclass. If every set in $\Gamma$ satisfies any of the Lebesgue measurability, Baire property, or the Ramsey property, then $\Gamma$ contains no MID family.
	\end{thm}
	
	\label{intro:lebesguemad}This result is notable for its uniformity, in contrast to the situation for MAD families. In Solovay's model, every set is Lebesgue measurable and there are no MAD families, while it has also been proved in \cite{horowitz2017madness} and \cite[Example 14.3.6]{larson2020geometric} that there is a model of $ \mathsf{ZF} + \mathsf{DC} $ where a MAD family exists and all sets of reals are Lebesgue measurable. This means that MID families are no less destructible than MAD families against classical regularity.
	
	Our fourth main result extends this investigation to generalized MAD notions that have recently attracted significant attention in the literature. We prove that:
	\begin{thm}[$\mathsf{ZF} + \mathsf{CC}_{\mathbb R}$]
		Let $\Gamma$ be a good pointclass such that every set in $\Gamma$ has the Ramsey property. Then $\Gamma$ contains no $\mathcal{ED}$-MAD families, no $\mathcal{ED}_{\mathrm{fin}}$-MAD families, and no $\finalpha{\alpha}$-MAD families for any $\alpha<\omega_1$.
	\end{thm}
	
	These generalizations are not straightforward, as each class represents a distinct combinatorial structure with its own properties. $\finalpha{\alpha}$-MAD families generalize the notion of almost disjointness using ideals of the form $\finalpha{\alpha}$ \cite{katvetov1968products}, while $\mathcal{ED}_{\mathrm{fin}}$-MAD and $\mathcal{ED}$-MAD families arise from the study of the eventually different ideal and its variants \cite{bartoszynski1995set}.
	
	The case of $\finalpha{\alpha}$-MAD families was previously proved in \cite{schrittesser2020ramsey} under the additional hypothesis $\unif$, where the definition of $\mathcal I$-MAD families for arbitrary ideals $\mathcal I$ on $\mathbb N$ is also given. Our result improves upon \cite{schrittesser2020ramsey} in two respects: we eliminate the need for the extra assumption $\unif$, and our proof is considerably simpler.

	\subsection{Overview}
	
	The structure of the paper is as follows. Section \ref{section:prelim} introduces the notation and reviews the basic definitions, including the Ramsey property, MAD families, MID families, Vitali sets, Hamel bases, and other relevant concepts. In Section \ref{section:conjecture}, we prove Mathias' conjecture. Section \ref{section:vitali} investigates the relationship between the Ramsey property and Vitali sets. Section \ref{section:mid} examines the connection between regularity properties and MID families. Section \ref{section:imad} studies the interplay between the Ramsey property and $\mathcal I$-MAD families. Finally, in Section \ref{section:problem}, we discuss some open problems and directions for future research.

	\section{Preliminaries}{\label{section:prelim}}
	
	In this section, we recall the necessary notation, definitions, and basic facts that will be used throughout the paper.

	\subsection{Notation}
	
	For any set $X$, we write $[X]^\omega$ for the collection of all infinite subsets of $X$,  $[X]^{<\omega}$ for the collection of all finite subsets of $X$, and $X^{<\omega}$ for the set of all finite sequences from $X$. Also, let $ \inffunc{X} $ be the set of countably infinite sequences $ \la x_0, \cdots, x_n,\cdots \ra$ of $ X $. For $ s\in \finfunc{X} $, define $ |s| $ as the length of $ s $. For a finite index set $ I=\{i_0,\cdots,i_k\} $ where $i_0<\cdots< i_k<|s| $, define $ s\uphar I $ as the restriction of $ s $ to $ I $, i.e. $ \la s(i_0), \cdots, s(i_k)\ra $.
	
	Let $\bbn$ be the set of non-negative integers. For $ a,b\in\bbn $, define $ [a,b] $ as $ \{n\in\bbn: a\leq n \leq b\} $, and let $ [a,b) $ be $ \{n\in\bbn: a\leq n < b\} $. We equip $[\mathbb N]^\omega$ with the subspace topology inherited from the product topology on $\mathcal P(\mathbb N) \cong \{0,1\}^{\mathbb N}$ via characteristic functions. 
	Equivalently, a basic open set in $[\mathbb N]^\omega$ is determined by a finite segment: for any finite sequence $s = \la s_0,\dots,s_{k-1}\ra$ with $s_0 < \cdots < s_{k-1}$, the set
	\[
	U_s = \{ x \in [\mathbb N]^\omega : x(0)=s_0,\dots,x(k-1)=s_{k-1} \}
	\]
	is basic open. This topology makes $[\mathbb N]^\omega$ a Polish space.
	
	For a Polish space $ X $, we also let
	\[
	K(X):=\{K\subseteq X: K \ \text{is compact}\} 
	\]
	denote the Polish space of compact subsets of $ X $ equipped with the Vietoris topology. The subbasic open sets are
	\[
	\{K\in K(X): K\subseteq U\}, \ \{K\in K(X): K \cap V\not= \emptyset\},
	\]
	where $ U $ and $ V $ are open subsets of $ X $. 
	
	For any $x\in[\mathbb N]^\omega$, we denote by $x(k)$ the $k$-th element of $x$ in increasing order, so that $x = \{x(0) , x(1) , x(2) , \cdots\}$ where $ x(0) < x(1) < x(2) < \cdots $.
	
	For sequences $s,t \in \mathbb{N}^{<\omega}$, we write $s^{\frown} t$ for their concatenation. When $n \in \mathbb{N}$, we identify $n$ with the sequence $\langle n \rangle$, and write $n^{\frown} s$ for $\langle n \rangle^{\frown} s$.
	
	We now introduce the tree sets that will serve as domains for higher-dimensional ideals. For each countable ordinal $\alpha < \omega_1$ with $\alpha \ge 1$, fix a non-decreasing sequence $\langle \gamma_n^\alpha : n \in \mathbb{N} \rangle$ of ordinals less than $\alpha$ such that $\bigcup_{n \in \mathbb{N}} (\gamma_n^\alpha + 1) = \alpha$ as follows. For successor ordinals $\alpha = \beta+1$ we take $\gamma_n^\alpha = \beta$ for all $n$, and for limit ordinals $\alpha$ we take $\gamma_n^\alpha \geq 1$ to be strictly increasing and cofinal in $\alpha$.
	
	For $ \alpha <\omega_1 $, the sets $S_\alpha \subseteq \mathbb{N}^{<\omega}$ are defined recursively as follows:
	\begin{itemize}
		\item[(1)] Define $ S_0 = \{\emptyset\} $.
		\item[(2)] For $\alpha \geq 1$, set
		\[
		S_\alpha = \{ n^{\frown} s : n \in \mathbb{N},\ s \in S_{\gamma_n^\alpha} \}.
		\]
	\end{itemize}
	Each $S_\alpha$ is countable and hence can be identified with $\mathbb{N}$ via a fixed bijection. In particular, $ S_1=\bbn $.
	
	For any $X \subseteq S_\alpha$, we introduce the following notation. For any finite sequence $s \in \mathbb{N}^{<\omega}$, let
	\[
	X(s) = \{ t \in \mathbb{N}^{<\omega} : s^{\frown} t \in X \}
	\]
	denote the \textbf{section} of $X$ at $s$. In particular, $X(\emptyset) = X$, and $ X(s) = \{\emptyset\} $ for $ s\in X $. We denote $X(\langle n \rangle) $ by $ X(n)$ for any $ n\in\bbn $.
	
	Let 
	\[
	\mathrm{dom}(X) = \{ n \in \bbn : X( n ) \neq \emptyset \}
	\]
	be the \textbf{domain} of $ X $, i.e. the set of indices with nonempty section. Note that for $ X\subseteq S_1 = \bbn $, we can identify $ \dom(X) = X $, since in this situation $ n\in X $ if and only if $ X(n)\not=\emptyset $.
	
	For any $ X\subseteq S_\alpha $, the set of initial segments \[
	T(X):= \{s\in \finfunc{\bbn} : \exists t\in X [s\subseteq t]\}
	\]
	is a well-founded tree, which admits a \textbf{rank function} as in \cite[Section~2.E]{kechris2012classical}. We denote the rank function of $ T(X) $ by $ \rho_X $, where for any $ s\in X $, $ \rho_X(s) = 0 $ and for any $ s\in T(X)\setminus X$ we have
	\[
	\rho_X(s) = \sup\{\rho_X(s^\frown n)+1: n\in \dom (X(s))\}.
	\]
	Also, we let $ \rho(X) = \sup\{\rho_X(s)+1: s\in T(X)\}$ denote the rank of the tree $ T(X) $.
	
	Parallel to the rank function, we can also consider the following notation recording the ordinal path of a sequence.
	
	For any $ s\in T(S_\alpha)$, denote $ \gamma(\alpha,\emptyset) = \alpha$ and
	\[
	\gamma(\alpha, s) = \gamma_{s(|s|-1)}^{\cdots^{\gamma_{s(0)}^\alpha}},
	\]
	so iteratively $ \gamma(\alpha, s^\frown \langle n\rangle ) = \gamma_{n}^{\gamma(\alpha, s)} $, and $ \gamma(\alpha,s^\frown t) = \gamma(\gamma(\alpha,s),t) $.
	
	The following are some useful facts whose proofs are omiited.
	
	\begin{fact}
		\leavevmode
		\begin{enumerate}[label=(\arabic*), font=\upshape]
			\item For any $ t\in S_\alpha $, we have $ \gamma(\alpha,t) =0$, and if moreover $ |t|\geq 1 $, we have $ \gamma(\alpha, t\uphar (|t|-1)) =1$.
			\item For any $ X\subseteq S_\alpha $ and $ t\in T(X)$, we have $ X(t)\subseteq S_{\gamma(\alpha, t)} $. So for any $ t\in X\subseteq S_\alpha $ with $ |t|\geq 1 $, $ X(t\uphar (|t|-1)) \subseteq \bbn$. In particular, for any $ t\in T(X) $ with $ \rho_X(t)=1 $, we have $ X(t)\subseteq \bbn $.
		\end{enumerate}
	\end{fact}

	\subsection{Definitions}
	
	\begin{defn}
		A relation is total if its domain is full. Now we define some well-known choice principles as follows. 
		\begin{itemize}
			\item[(1)] Let $\mathsf{CC}_{\mathbb{R}}$ be the assertion that every countable family of nonempty subsets of $\mathbb{R}$ has a choice function.
			\item[(2)] Let $\mathsf{DC}_{\mathbb{R}}$ be the assertion that for every nonempty $X \subseteq \mathbb{R}$ and every total relation $R \subseteq X \times X$, there exists an infinite sequence $\langle x_n: n\in \bbn \rangle$ with $x_n \mathrel{R} x_{n+1}$ for all $n$.
		\end{itemize}
	\end{defn}
	
	It is well known that $\mathsf{DC}_{\mathbb{R}}$ implies $\mathsf{CC}_{\mathbb{R}}$, but the converse does not hold in $ \mathsf{ZF} $. Also, recall that a set $X$ is infinite if $X$ is not bijective to any natural number, and is \textbf{Dedekind-infinite} if there is an injection from $\mathbb{N}$ to $X$. Clearly, a Dedekind-infinite set is an infinite set. On the other hand, $\mathsf{CC}_{\mathbb{R}}$ implies that infinite subsets of $\mathbb{R}$ are Dedekind-infinite. 
	
	\begin{defn}
		A set $X \subseteq [\mathbb N]^\omega$ is said to have the \textbf{Ramsey property} if there exists an infinite set $H \in [\mathbb N]^\omega$ such that either $[H]^\omega \subseteq X$ or $[H]^\omega \cap X = \emptyset$. We say that $H$ is homogeneous for $X$.
	\end{defn}
	
	Note that for any $A\in\infsub{\bbn}$, we can define Ramsey property of subsets in $\infsub{A}$ analogous to $\infsub{\bbn}$. Let $\hat{A}$ be the unique strictly increasing function from $\bbn$ onto $A$. The bijection $[\hat{A}]: \infsub{\bbn}\rtar \infsub{A}$ which sends $ z $ to $ \hat{A}''z $ preserves Ramsey property in both directions.
	
	\begin{defn}
		For any $a\in \finsub{\bbn}, A\in\infsub{\bbn}$ such that $\max a < \min A$, we define $[a,A] = \{B\in \infsub{\bbn}: a\subseteq B\subseteq a\cup A\}$. A set $X \subseteq  [\mathbb N]^\omega$ is \textbf{completely Ramsey} if for every such $a$ and $A$, there exists $H \in [A]^\omega$ such that either $[a, H] \subseteq X$ or $[a, H] \cap X = \emptyset$. We say that $[a,H]$ is homogeneous for $X$.
	\end{defn}
	
	\begin{defn} A \textbf{good pointclass} is a collection $\Gamma$ of subsets of Polish spaces that contains all clopen sets and is closed under countable intersections, finite products, countable unions, continuous images, and continuous preimages. A function $f: X \to Y$ between Polish spaces is said to be \textbf{$\Gamma$-measurable} if for every open set $U \subseteq Y$, the preimage $f^{-1}[U]$ belongs to $\Gamma$. 
	\end{defn}	
	
	Note that $\Gamma$ may contain subsets of different Polish spaces simultaneously. When we speak of the pointclass $\Gamma$ having the Ramsey property, we mean that every set in $ \Gamma\uphar \infsub{\bbn} $ has the Ramsey property. When we say a set in $ \Gamma $ is Lebesgue measurable, for our purpose, we mean that the set is a subset of $ \bbr $ or $ \prst(\bbn) $.
	
	\begin{defn}
		An \textbf{almost disjoint family} is a collection $\mathcal A\subseteq[\mathbb N]^\omega$ such that $A\cap B$ is finite for all distinct $A,B\in\mathcal A$. An almost disjoint family $\mathcal A$ is called \textbf{maximal} $($or a \textbf{MAD family}$)$ if for every $x\in[\mathbb N]^\omega$, there exists $A\in\mathcal A$ such that $x\cap A$ is infinite.
	\end{defn}
	
	Equivalently, $\mathcal A$ is MAD if it is an almost disjoint family and is not properly contained in any other almost disjoint family.
	
	\begin{defn}
		A \textbf{Vitali set} is a subset $V\subseteq\mathbb{R}$ such that for every $x\in\mathbb{R}$ there exists a unique $v\in V$ with $x-v\in\mathbb{Q}$.
	\end{defn}
	
	\begin{defn}
		A \textbf{Hamel basis} is a basis for $\mathbb{R}$ as a vector space over $\mathbb{Q}$. That is, a set $H\subseteq\mathbb{R}$ such that:
		\begin{itemize}
			\item[(1)] Any finite $\mathbb{Q}$-linear combination of elements in $H$ that equals zero must have all coefficients zero;
			\item[(2)] Every $x\in\mathbb{R}$ can be expressed as a finite $\mathbb{Q}$-linear combination of elements in $H$.
		\end{itemize}
	\end{defn}
	
	For any non-empty subset $ B\subseteq \bbr $, we also define its \textbf{span} in $ \bbr $ to be the set of all finite $\mathbb{Q}$-linear combinations, and denote it by $ \operatorname{span}(B) $.
	
	\begin{defn}
		A family $\mathcal A \subseteq \mathcal P(\mathbb N)$ is \textbf{independent} if for any two disjoint finite subfamilies $\mathcal F, \mathcal G \subseteq \mathcal A$, the Boolean combination
		\[\bigcap \mathcal F\setminus \bigcup \mathcal G\]
		is infinite. A \textbf{maximal independent $($MID$)$ family} is an independent family that is maximal under inclusion.
	\end{defn}

	\begin{defn}
		Let $S$ be a set. A collection $\mathcal I \subseteq \mathcal P(S)$ is called a (proper) \textbf{ideal} on $S$ if:
		\begin{itemize}
			\item[(1)] $\emptyset \in \mathcal I$ and $S \notin \mathcal I$;
			\item[(2)] If $A\in\mathcal I$ and $B\subseteq A$, then $B\in\mathcal I$;
			\item[(3)] If $A,B\in\mathcal I$, then $A\cup B\in\mathcal I$.
		\end{itemize}
		A (proper) \textbf{filter} on $S$ is a collection $\mathcal F \subseteq \mathcal P(S)$ such that its dual 
		\[
		\mathcal{F}^\ast = \{ S \setminus A : A \in \mathcal F \}
		\]
		is an  ideal on $S$. 
		
		For an ideal $\mathcal I$, we denote by $\mathcal I^+ = \mathcal P(S) \setminus \mathcal I$ the collection of $\mathcal I$-\textbf{positive} sets.
	\end{defn}	
	
	In this paper, for our purposes, all ideals are assumed to contain all finite subsets of their underlying set. This convention ensures that our ideals are proper and nontrivial.
	
	\begin{defn}Let $\mathcal I$ be an ideal on $\mathbb{N}$.
		\begin{itemize}
			\item[(1)] The sets $A,B\subseteq\mathbb{N}$ are \textbf{$\mathcal I$-almost disjoint} if $A\cap B\in\mathcal I$.
			\item[(2)] A family $\mathcal A\subseteq \idl{I}^+$ is an \textbf{$\mathcal I$-almost disjoint family} if its members are pairwise $\mathcal I$-almost disjoint.
			\item[(3)] An \textbf{$\mathcal I$-MAD family} is an $\mathcal I$-almost disjoint family that is maximal under inclusion: for every $X\in \idl{I}^+$, there exists $A\in\mathcal A$ such that $X\cap A \in \idl{I}^+$.
		\end{itemize}
	\end{defn}
	
	\begin{defn} The \emph{eventually different ideal} $\mathcal{ED}$ is an ideal on $\omega \times \omega$ defined by
		\[
		\mathcal{ED} = \left\{ A \subseteq \mathbb N \times \mathbb N : \exists n \ \forall m \ge n \left[ |\{k : (m,k) \in A\}| \le n \right] \right\},
		\]
		and $\mathcal{ED}_{\text{fin}}$ is defined as the restriction of $\mathcal{ED}$ to the set 
		\[
		\Delta = \{ (n,m) \in \mathbb N \times \mathbb N : m \le n \}.
		\]
	\end{defn}
	
	For $ A\subseteq \bbn\times \bbn$ or $ A \subseteq \Delta  $, and $ n\in\bbn $, we also denote $ A(n) = \{m: (n,m)\in A\} $ and
	\[
	\dom(A) =\{n\in \bbn: A(n)\neq \emptyset \}.
	\]
	
	\begin{defn}
		For each countable ordinal $\alpha < \omega_1$, we define the \textbf{iterated Fr\'echet ideal} $\mathrm{Fin}^\alpha$ on $S_\alpha$ recursively as follows.
		
		\begin{itemize}
			\item[(1)] For $\alpha = 0$, let
			\[
			\mathrm{Fin}^0 =\{\emptyset\}.
			\]
			
			\item[(2)] For $\alpha \ge 1$, a set $X \subseteq S_\alpha$ belongs to $\finalpha{\alpha}$ if and only if
			\[
			\{ n \in \mathbb{N} : X(n) \notin \finalpha{\gamma_n^\alpha} \} \text{~is ~finite}.
			\]
		\end{itemize}
		
		For example, $ \finalpha{1} $ is the set of all finite subsets of $ \bbn $. Note that for $ X\subseteq S_1 $, we have $ X(n)\in \finalpha{0} $ if and only if $ n\notin X $. For our specific purposes, we call a set \textbf{Fr\'echet small} if it is in the ideal $ \finalpha{\alpha} $ for some $ \alpha $.
	\end{defn}

	\section{Proof of Mathias' conjecture}{\label{section:conjecture}}
	
	We prove Theorem~\ref{theorem:introduction} in the following slightly more general form.
	
	\begin{thm}[$\zfaxiom$]\label{thm:mad-ramsey}
		Let $\Gamma$ be a good pointclass. If every set in $\Gamma$ has the Ramsey property, then there is no Dedekind-infinite MAD family in $\Gamma$.
	\end{thm}
	
	\begin{proof}
		Let $\mathcal A \in \Gamma$ be a Dedekind-infinite MAD family, and fix a sequence $\{A_n: n\in \mathbb{N}\}$ of distinct elements in $\mathcal A$.
		By replacing each $A_n$ with $A_n\setminus \bigcup_{i<n} A_i$ if necessary, we may assume without loss of generality that the sets $A_n$ are pairwise disjoint. We can further assume that $ A_n(0)<A_{n+1}(0) $ for any $ n\in\bbn $.
		
		For $x\in [\mathbb N]^\omega$, define
		\[
		\Phi(x)=\{A_{x(2n)}(x(2n+1)): n\in\mathbb{N}\}. \tag{$\ast$}
		\]
		Figure \ref{fig:tildefin} illustrates this definition.
		
		\begin{figure}[h]
			\centering
			\begin{tikzpicture}[scale=0.9]
				\foreach \n in {0,1,2,3,4,5,6,7} {
					\draw[gray!30, thick] (\n*1.2, -2.5) -- (\n*1.2, -8);
				}
				
				\foreach \n in {0,1,2,3,4,5,6,7} {
					\foreach \k in {6,...,15} {
						\fill[blue] (\n*1.2, -\k*0.5) circle(1.5pt);
					}
				}
				
				\draw[->, thick] (-0.8, -8) -- (-0.8, -2.5); \node[left] at (-0.8, -8) {$0$};
				\node[left] at (-0.8, -2.5) {$\infty$};
				
				\node[above] at (0, -8.7) {$A_0$};
				\node[above] at (1.2, -8.7) {$A_1$};
				\node[above] at (2.4, -8.7) {$A_2$};
				\node[above] at (3.6, -8.7) {$A_3$};
				\node[above] at (4.8, -8.7) {$A_4$};
				\node[above] at (6.0, -8.7) {$A_5$};
				\node[above] at (7.2, -8.7) {$\cdots$};
				\node[above] at (8.4, -8.7) {$A_m$};
				
				\fill[red] (0, -13*0.5) circle(3pt);
				\node at (0, -13*0.5) [right] {\small $A_{0}(2)$};
				
				\fill[red] (3.6, -11*0.5) circle(3pt);
				\node at (3.6, -11*0.5) [right] {\small $A_3(4)$};
				
				\fill[red] (8.4, -9*0.5) circle(3pt);
				\node at (8.4, -9*0.5) [right] {\small $A_m(k)$};
				
			\end{tikzpicture}
			\caption{Definition of $\Phi(x)$ for $x=\{0,2,3,4,\cdots,m,k,\cdots\}$}
			\label{fig:tildefin}
		\end{figure}
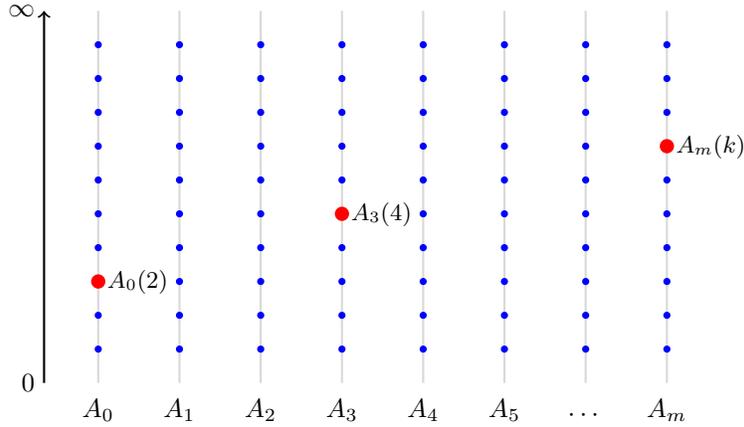
		
		\begin{claim}{\label{claim:dim1hide}}
			For every $x\in [\mathbb N]^\omega$,  there exist $y\in [x]^\omega$	and $A\in \mathcal A$
			such that $ \Phi(y)\subseteq A$.
		\end{claim}
		
		\begin{claimproof}{\ref{claim:dim1hide}}
			Since $\mathcal A$ is a MAD family and $\Phi(x)$ is an infinite subset of $\mathbb{N}$,  there exists some $A\in \mathcal A$ such that $\Phi(x)\cap A$ is infinite. 
			By the definition of $\Phi(x)$, we have $A\not= A_n$ for all $n\in \mathbb{N}$.
			Enumerate $\Phi(x)\cap A$ as:
			\[ \{A_{x(2n_k)}(x(2n_k+1)): k\in\mathbb{N} \}\]
			Now, define: 
			\[
			y=\{x(2n_k):k\in\mathbb{N}\}
			\cup\{x(2n_k+1):k\in\mathbb{N}\}.
			\]  
			Then, we have \[\Phi(y)=\{A_{y(2n)}(y(2n+1)):n\in\mathbb{N}\}=\Phi(x)\cap A\subseteq A,\]
			as required.
		\end{claimproof}
		
		\begin{claim}{\label{claim:dim1split}}
			For every $x\in [\mathbb N]^\omega$,  there exist $y\in [x]^\omega$	and $A\in \mathcal A$ such that both $ A\cap\Phi(y)$ and $ \Phi(y)\setminus A$ are infinite.
		\end{claim}
		
		\begin{claimproof}{\ref{claim:dim1split}}
			By the previous claim, we can assume that $\Phi(x) \subseteq A$, for some $A \in \mathcal A$.
			As $A\not= A_n$ for all $n\in \mathbb{N}$, $A\cap A_n$ is finite for all $n\in\mathbb{N}$.
			
			Let $x_e = \{x(2n) : n \in \mathbb{N}\}$ be the set of even-indexed elements of $x$. We construct $ y\in \infsub{x} $ by recursively defining $ y(i)$,  $y(i+1)$, $y(i+2)$, $y(i+3)$ for every $ i\in\{4k:k\in\bbn\} $ as follows. We assume $ y(-1)=-1 $ for convenience.
			\begin{itemize}
				\item[(1)] Define $ y(i) = \min\{u\in x_e: u > y(i-1)\} $.
				\item[(2)] Define $ y(i+1) = \min\{u\in x : A_{y(i)}(u)\notin A \}$. This is possible as $ A_{y(i)}\cap A $ is finite.
				\item[(3)] Define $y(i+2) = \min\{u \in x_e : u > y(i+1)\}$.
				\item[(4)] Define $y(i+3) = \min\{u\in x: u > y(i+2) \}$. This ensures that  
				$A_{y(i+2)}\big(y(i+3)\big) \in A$.
			\end{itemize}
			Clearly, $\Phi(y)=\big\{A_{y(2i)}\big(y(2i+1)\big):i\in\mathbb{N}\big\}$ has the required property.
		\end{claimproof}
		
		Let us define
		\[
		P=\{z\in [\mathbb N]^\omega: \exists A\in \mathcal A \ [\Phi(z)\subseteq A ] \}.
		\] 
		
		As
		\[
		A_{2n}(0) \leq A_{x(2n)}(0) < A_{x(2n)}(x(2n+1)),
		\]
		the quantifier $ n $ has only finitely many choices when verifying $ p\notin \Phi(x)$. We have that $ \Phi $ is continuous. Since $\Gamma$ is a good pointclass, we have $P\in \Gamma$. Now we show $P$ is not Ramsey. Otherwise, one of the following cases occurs:
		\begin{itemize}
			\item[(1)] Suppose that there exists $x \in [\mathbb N]^\omega$ such that $[x]^\omega\cap P = \emptyset$.
			
			By Claim \ref{claim:dim1hide}, there exist $y \in [x]^\omega$ and $A \in \mathcal A$ with $\Phi(y) \subseteq A$, i.e., $y \in P$.
			This contradicts the assumption that $[x]^\omega \cap P=\emptyset$.
			\item[(2)] Suppose there exists $x \in [\mathbb N]^\omega$ such that $[x]^\omega \subseteq P$.
			
			By Claim \ref{claim:dim1split}, there exist $y \in [x]^\omega$ and $A \in \mathcal A$ such that both $A\cap \Phi(y)$ and $\Phi(y)\setminus A$ are infinite. Note that $ y\in P $, so there is $ C\in \idl{A} $ such that $ \Phi(y)\subseteq C $. As $A\cap \Phi(y)$ is infinite, we have that $ A\cap C $ is also infinite, and thus $ A=C $ by almost disjointness. This means $ \Phi(y)\subseteq A $, a contradiction.
		\end{itemize}
		
		Both cases lead to a contradiction, completing the proof.
	\end{proof}
	
	Note that Theorem \ref{thm:mad-ramsey} has several classical consequences. In particular, it gives an additional proof that there are no analytic infinite MAD families.
	
	\begin{cor}[\cite{mathias1977happy}]
		There are no analytic infinite MAD families.
	\end{cor}
	\begin{proof}
		Apply Theorem \ref{theorem:introduction} where $\Gamma$ is the pointclass of analytic sets.
	\end{proof}
	
	As a further application, we recover Törnquist's theorem on Solovay's model.
	
	\begin{cor}[\cite{tornquist2018definability}]
		There are no infinite MAD families in Solovay's model.
	\end{cor}
	\begin{proof}
		Solovay's model \cite{solovay1970model} satisfies $\zfaxiom + \mathsf{CC}_{\mathbb{R}}$, and Mathias \cite{mathias1969generalization} proved that every set of reals in this model has the Ramsey property. Apply Theorem \ref{theorem:introduction} where $\Gamma$ is the pointclass of all sets.
	\end{proof}
	
	Assuming the \emph{Projective Determinacy} $ \pdaxiom $, we also obtain:
	
	\begin{cor}[\cite{neeman2018happy}]
		$\pdaxiom$ implies there are no projective infinite MAD families.
	\end{cor}
	\begin{proof}
		Under $\pdaxiom$, every projective set is completely Ramsey \cite{harrington1981determinacy}. Apply Theorem \ref{theorem:introduction} where $\Gamma$ is the pointclass of all projective sets.
	\end{proof}

	\section{Ramsey property and Vitali set}{\label{section:vitali}}
	
	We now turn to Vitali sets, another classical object of study. We will use the following facts about the Ramsey property.
	
	\begin{prop}\label{prop:ramsey-continuous}
		Let $\Gamma$ be a good pointclass.
		\begin{enumerate}[label=(\arabic*), font=\upshape]
			\item[(1)] All sets in $\Gamma$ are Ramsey if and only if all sets in $\Gamma $ are completely Ramsey.
			\item[(2)] \textnormal{($\zfaxiom + \mathsf{DC}_{\mathbb R}$)} If all sets in $\Gamma$ are Ramsey, then for any $\Gamma$-measurable $f: \infsub{\bbn}\rtar Y$ where $Y$ is Polish, there is an $H\in \infsub{\bbn}$ such that $f\uphar\infsub{H}$ is continuous.
		\end{enumerate}
	\end{prop}
	
	\begin{proof}
		\leavevmode
		\begin{itemize}
			\item[(1)] One direction is trivial since completely Ramsey sets are Ramsey. Conversely, for any $ X\in\Gamma$ and $ [a,A]$, define $f:\infsub{A}\rtar \infsub{\bbn}$ such that $f(T) = a\sqcup T$. Then $f^{-1}[X]$ is in $\Gamma\uphar\infsub{A}$ by the closure property of $ \Gamma $. So $f^{-1}[X]$ is Ramsey in $ \infsub{A} $ and there is $B\in\infsub{A}$ homogeneous for $f^{-1}[X]$. Hence $[a,B]$ is homogeneous for $X$.
			\item[(2)] This proposition is similar to \cite[Exercise~19.19]{kechris2012classical}. Let $\{U_k:k\in\bbn\}$ be a basis of $Y$. Since $ f $ is $ \Gamma $-measurable, we have $f^{-1}[U_k]\in \Gamma$, so by (1) they are completely Ramsey. Now, we follow an argument similar to the fusion argument in \cite[Lemma~19.17]{kechris2012classical}, i.e. 
			\begin{itemize}
				\item[i.] Let $A_{0}=\bbn, B_{0}=\emptyset$.
				\item[ii.] For any $k\geq 0$, choose $A_{k+1}\subseteq A_k\backslash(1+ \max B_k)$ (let $\max \emptyset = -1$ for convenience) such that for any $b\subseteq B_k $, $[b,A_{k+1}]$ is homogeneous for $f^{-1}[U_k]$, and let $n_{k+1} = \min A_{k+1}, B_{k+1}=B_k\cup \{n_{k+1}\}$.
			\end{itemize}  
			Let $C=\{n_k:k>0\}$. Then for any $k\in\bbn$, $f^{-1}[U_k]\cap \infsub{C}$ is a union of sets of the form $[\{n_{k_0}, \dots, n_{k_m}\}, C\backslash n_{k_{m}+1}]$, which is open in $\infsub{C}$ with the usual topology, and $\dcaxiom_\bbr$ suffices to choose all sets $A_{k+1}$.
		\end{itemize}		
	\end{proof}
	
	We can now prove the main result of this section.
	
	\begin{thm}[$\zfaxiom + \mathsf{DC}_{\mathbb R}$]\label{thm:ramsey-no-vitali}
		Let $\Gamma$ be a good pointclass. If every subset of $[\mathbb N]^\omega$ in $\Gamma$ has the Ramsey property, then no Vitali set is in $\Gamma$. 
	\end{thm}
	
	\begin{proof}
		Suppose otherwise, and let $V \in \Gamma$ be a Vitali set. 
		For any $x \in \mathbb R$, define $v(x)$ to be the unique element in $ V $ such that $x - v(x) \in \bbq$. Note that $v$ is $\Gamma $-measurable.
		
		We define the following encoding function $\encfunc: \infsub{\bbn}\rtar\bbr$ to record combinatorics on $ \infsub{\bbn} $. It is defined by 
		\[
		\encfunc(x):=\sum_{i=0}^{\infty}\frac{1}{2^{x(i)+1}},
		\]
		which is continuous. 
		
		Let $ 0_n $ be the sequence of $n$-many $0$'s. By \cite[~Exercise 2.19]{bukovsky2011structure}, the value $\encfunc(x)$ is irrational if and only if its binary expansion
		\[
		0.0_{x(0)}10_{x(1)- x(0) -1}1\cdots 0_{x(n+1)- x(n) -1}1\cdots
		\]
		is not periodic, where a binary expansion $\{x_i\}_{i=1}^\infty$ of a real is called periodic if 
		there are natural numbers $k, l$ such that $x_{i+l} = x_i$ for every $i > k$.
		
		Consider $f = v \circ \encfunc: [\mathbb N]^\omega \to \bbr$, which is also $\Gamma $-measurable. By Proposition \ref{prop:ramsey-continuous}(2), there exists $H \in [\mathbb N]^\omega$ such that $f \upharpoonright [H]^\omega$ is continuous.	
		
		\begin{claim}{\label{claim:encodecisconst}}
			The function $f$ is constant on $[H]^\omega$.
		\end{claim}
		
		\begin{claimproof}{\ref{claim:encodecisconst}}
			If not, there exist $ x_1, x_2 \in \infsub{H} $ with $ f(x_1) \not= f(x_2) $. Fix disjoint open sets $ W_1, W_2 \subseteq \bbr$ such that $ f(x_i) \in W_i$ for $i=1,2 $. By continuity, there exist basic open subsets $ U_{s_1}, U_{s_2} \subseteq \infsub{H} $ containing $ x_1, x_2 $ such that the image $ f'' U_{s_i} $ is contained in $ W_i$ for $i=1,2 $. We can assume $ |s_1| = |s_2| $. Consider
			\[
			z = \{s_1(i) :i< |s_1|\} \sqcup \{n\in x_2: n> s_1(|s_1|-1)\}.
			\]
			By the property of the encoding function, we have $ \encfunc(z)-\encfunc(x_2)\in \bbq $, so by definition $ f(z) = f(x_2)\in W_2 $. Also, we have $ z\in U_{s_1} $, so $ f(z)\in W_1$, a contradiction.
		\end{claimproof}
		
		Now, note that we can recursively define $y \in [H]^\omega$ such that $H\setminus y$ is infinite and the sequence
		\[
		\{ (H\setminus y)(i+1) - (H\setminus y)(i) : i \in \bbn\}
		\]
		is strictly increasing and thus not periodic, which means $ \encfunc(H\setminus y) $ is irrational. Since $\encfunc(H) - \encfunc(y) = \encfunc(H\setminus y)$, it follows that $\encfunc(H)-\encfunc(y) \notin \mathbb Q$. Consequently, $f(H) \neq f(y)$, contradicting the constancy of $f$ on $[H]^\omega$.
		
		Therefore, no Vitali set belongs to $\Gamma$.
	\end{proof}
	
	The following folklore proposition shows that the existence of a Hamel basis implies the existence of a Vitali set within the same pointclass. This fact, combined with Theorem \ref{thm:ramsey-no-vitali}, will allow us to conclude that the Ramsey property also precludes the existence of Hamel bases.
	
	\begin{prop}[\cite{beriashvili2018hamel}]\label{prop:hamel-vitali}Assume $ \zfaxiom $ and let $\Gamma$ be a good pointclass. If $\Gamma$ contains a Hamel basis, then $\Gamma$ also contains a Vitali set.
	\end{prop}
	
	\begin{proof}
		Let $B \in \Gamma$ be a Hamel basis. For convenience, when we express a non-zero real as a finite combination, we always assume the coefficients are non-zero and the spanning vectors are distinct.
		
		Write $1 = \sum_{k=1}^n q_k z_k$ and set $V = \operatorname{span}(B \setminus \{z_1\})$. Then $1 \notin V$. We show that $ V $ is a Vitali set in $ \Gamma $.
		
		For any $x \in \mathbb R$, write $x = \sum p_k w_k$ with $w_k \in B$. If $z_1$ does not appear as some $ w_k $, then $x \in V$. Otherwise, let $p$ be the coefficient of $z_1$, and we have $x - \frac{p}{q_1}\cdot 1 \in V$. Hence every real differs from some element of $V$ by a rational.
		
		If $v_1, v_2 \in V$ with $v_1 - v_2 \in \mathbb Q$, since $ V $ is a span we have $v_1 - v_2 \in V \cap \mathbb Q$. We show that $V \cap \mathbb Q = \{0\}$. If there is some $r \in V \cap \mathbb Q \setminus \{0\}$, let $r = \sum r_i u_i$ with $u_i \in B \setminus \{z_1\}$. This means $1 = \sum \frac{r_i}{r} \cdot u_i$, contradicting the uniqueness of the basis representation of $ 1 $. Thus $V \cap \mathbb Q = \{0\}$, which means $v_1 = v_2$. Hence $V$ is a Vitali set.
		
		Now we show $ V\in\Gamma $. Denote
		\[
		V_n = \left\{\sum_{i=1}^{n} q_i z_i: \forall i\leq n \ \left[ q_i \in \bbq \  \wedge \  z_i \in B \setminus \left\{z_1\right\}\right]\right\},
		\]
		which is the image of the linear combination function. Since this function is continuous and the domain is in $ \Gamma $, we have $ V_n \in \Gamma$. As $ V = \bigcup_{n\in\bbn} V_n $, this finishes the proof.
	\end{proof}
	
	Combining Proposition \ref{prop:hamel-vitali} with Theorem \ref{thm:ramsey-no-vitali}, we immediately obtain the following consequence for Hamel bases.
	
	\begin{cor}[$\zfaxiom + \mathsf{DC}_{\mathbb R}$]
		Let $\Gamma$ be a good pointclass. If every set in $\Gamma$ has the Ramsey property, then there is no Hamel basis in $\Gamma$.
	\end{cor}

	\section{Regularity properties and MID family}{\label{section:mid}}
	
	In this section, we prove that within a given good pointclass $\Gamma$, under any of the three classical regularity properties i.e. Lebesgue measurability, Baire property, or Ramsey property, there is no maximal independent family. The proof proceeds by first showing that the existence of an infinite MID family forces a decomposition of $\mathcal P(\mathbb N)$ into countably many filters and ideals from $\Gamma$, and then using this decomposition to derive a contradiction from each regularity assumption.
	
	\begin{defn}
		Let $\mathcal A \subseteq \mathcal P(\mathbb N)$ be a MID family from $\Gamma$. 
		Let $\mathcal U$ consist of all finite disjoint unions of basic open sets of $[\mathbb N]^\omega$. Note that $ \idl{U} $ is countable.
		
		For $x \in [\mathbb N]^\omega$, define
		\[
		\begin{aligned}
			\idl{C}(x) = \{ (F,G) \in K([\mathbb N]^\omega) \times K([\mathbb N]^\omega) : & \ F,G \in [\mathcal A]^{<\omega},\ F \cap G = \emptyset, \\
			& \sigma(F,G) \setminus x \text{ is finite } \},
		\end{aligned}
		\]
		\[
		\begin{aligned}
			\idl{D}(x) = \{ (F,G) \in K([\mathbb N]^\omega) \times  K([\mathbb N]^\omega) : & \ F,G \in [\mathcal A]^{<\omega},\ F \cap G = \emptyset, \\
			& \sigma(F,G) \cap x \text{ is finite } \},
		\end{aligned}
		\]
		where
		\[
		\sigma(F,G) = \bigcap_{a \in F} a \cap \left( \mathbb N \setminus \bigcup_{b \in G} b \right).
		\]
		
		For disjoint $U, V \in \mathcal U$, define	
		\[
		\begin{aligned}
			\mathcal F(U, V) &= \{ x \in [\mathbb N]^\omega : \exists (F,G) \in \langle U\rangle \times \langle V\rangle,  (F,G) \in \idl{C}(x) \},\\
			\mathcal I(U, V) &= \{ x \in [\mathbb N]^\omega : \exists (F,G) \in \langle U\rangle \times \langle V\rangle , (F,G) \in \idl{D}(x) \} \cup [\mathbb N]^{<\omega},
		\end{aligned}
		\]
		where $\langle U\rangle \times \langle V\rangle$ denotes the basic open set
		\[
		\{(F, G)\in K([\mathbb N]^\omega) \times K([\mathbb N]^\omega):F\subseteq U, G\subseteq V\}.
		\]
	\end{defn}
	
	We begin with the following key decomposition.
	
	\begin{lem}\label{lemma:mid-union} Let $\Gamma$ be a good pointclass. If $\Gamma$ contains a MID family, then $\mathcal P(\mathbb N)$ is a countable union of filters and ideals from $\Gamma$.
	\end{lem}
	
	\begin{proof}
		The proof proceeds in three steps. First, for any two disjoint open sets $U,V$, we show that $\mathcal F(U,V)$ and $\mathcal I(U,V)$ are a filter and an ideal, respectively. Second, we prove that every real $x\in \prst(\bbn)$ belongs to some $\mathcal F(U,V)$ or $\mathcal I(U,V)$. Finally, we verify that these filters and ideals belong to $\Gamma$.
		
		\begin{claim}{\label{claim:theyareideals}}
			If $ \idl{F}(U,V)$ and $ \idl{I}(U,V) $ are non-empty, then $\idl{F}(U, V)$ is a proper filter and $\idl{I}(U, V)$ is a proper ideal.
		\end{claim}
		
		\begin{claimproof}{\ref{claim:theyareideals}}
			We verify the ideal properties for $\mathcal I(U,V)$; the filter properties for $\mathcal F(U,V)$ are analogous. In fact, $ \idl{F}(U,V) $ is the dual filter of $ \idl{I}(U,V) $.
			\begin{itemize}
				\item[(1)] If $x \in \mathcal I(U, V)$ and $y \subseteq x$, then $y \in \mathcal I(U, V)$.
				
				This is clear when $x$ is finite. If $x$ is infinite, pick $(F,G) \in \langle U\rangle \times \langle V\rangle$ with $(F,G) \in \idl{D}(x)$. Then $\sigma(F,G) \cap x$ is finite, and since $y \subseteq x$, the same $(F,G)$ witnesses $y \in \mathcal I(U, V)$.
				
				\item[(2)] For any $x_1, x_2 \in \mathcal I(U, V)$, we show that $ x_1 \cup x_2 \in \idl{I}(U,V) $. 
				
				Pick $(F_i,G_i) \in \langle U\rangle \times \langle V\rangle$ with $\sigma(F_i,G_i) \cap x_i $ finite. Let $F = F_1 \cup F_2$, $G = G_1 \cup G_2$. Then $(F,G) \in \langle U\rangle \times \langle V\rangle$ and $ F\cap G =\emptyset$. We have that
				\begin{alignat*}{2}
					& \phantom{=}& &\phantom{\big(}\sigma\left(\mathmakebox[\widthof{$F_1$}][c]{F} , \mathmakebox[\widthof{$G_1$}][c]{G} \right) \cap \left(x_1 \cup x_2\right) \\
					&=& &\big(\sigma\left(F_1,G_1\right) \cap \sigma\left(F_2,G_2\right)\big) \cap \left(x_1 \cup x_2\right) \\
					&=& \big[&\big(\sigma\left(F_1,G_1\right) \cap \sigma\left(F_2,G_2\right)\big) \cap x_1\big] \\
					&\phantom{=}\!\cup& \big[&\big(\sigma\left(F_1,G_1\right) \cap \sigma\left(F_2,G_2\right)\big) \cap x_2\big] \\
					&\subseteq& &\big(\sigma\left(F_1,G_1\right) \cap x_1\big) \cup \big(\sigma\left(F_2,G_2\right) \cap x_2\big),
				\end{alignat*}
				which is finite. Hence $x_1 \cup x_2 \in \mathcal I(U, V)$.
				\item[(3)] We show that $\mathbb N \notin \mathcal I(U, V)$. If there is any witness $(F,G)$, then $\sigma(F,G) \cap \mathbb N = \sigma(F,G)$ is finite. But by independence $\sigma(F,G)$ is infinite.
			\end{itemize}
		\end{claimproof}
		
		Now, we show that every $x \in \prst(\bbn)$ lies in some $\mathcal F(U, V)$ or $\mathcal I(U, V)$ with $ U,V \in \idl{U} $. If $x \in [\mathbb N]^{<\omega}$, then $x \in \mathcal I(U, V)$ for any $U,V \in \mathcal U$ by definition. 
		Assume $x \in [\mathbb N]^\omega$. By the maximality of the MID family $\mathcal A$, there exist finite disjoint $F,G \subseteq \mathcal A$ such that either $\sigma(F,G) \setminus x $ or $\sigma(F,G) \cap x $ is finite. We can find disjoint $U, V \in \mathcal U$ with $F \subseteq U$, $G \subseteq V$. Then $(F,G) \in \langle U\rangle \times \langle V\rangle$, and
		\[
		\left\{
		\begin{aligned}
			\text{if }\sigma(F,G) \setminus x \text{ is finite, } &\text{then } x \in \mathcal F(U, V);\\
			\text{if }\sigma(F,G) \cap x \text{ is finite, } &\text{then } x \in \mathcal I(U, V).
		\end{aligned}
		\right.
		\]
		Since $\mathcal U$ is countable, we have that
		\[
		\mathcal P(\mathbb N) = \bigcup_{\substack{U, V \in \mathcal U \\ U \cap V = \emptyset}} \big( \mathcal F(U, V) \cup \mathcal I(U, V) \big)
		\]
		is a countable union of filters and ideals.
		
		Finally, we show that $\mathcal F(U,V)$ and $\mathcal I(U,V)$ belong to $\Gamma$. We prove that $\mathcal F(U, V) \in \Gamma$. Since the pointwise complement $ x\mapsto \bbn\setminus x $ is a homeomorphism of $ \prst(\bbn) $, we also have that $\mathcal I(U, V) \in \Gamma$.
		
		Define
		\[
		\begin{aligned}
			Z=\big\{ \lb \lb F, G \rb, x\rb &\in K \lb \infsub{\bbn} \rb \times K\lb \infsub{\bbn} \rb \times \prst\lb \bbn\rb:\\
			\lb F, G\rb &\in \langle U\rangle \times \langle V\rangle\cap \finsub{\idl{A}} \times \finsub{\idl{A}} \  \wedge\\
			F\cap G &=\emptyset \ \wedge \ \sigma\lb F,G \rb \cap x \text{ is finite} \big\}.
		\end{aligned}
		\]
		Note that the function $ f_n: \langle x_1, \cdots , x_n \rangle \mapsto \{x_1, \cdots, x_n\} $ is a continuous function from $ \lb \infsub{\bbn} \rb^n $ to $ K(\infsub{\bbn}) $, so we have that
		\[
		\langle U\rangle \times \langle V\rangle\cap \finsub{\idl{A}} \times \finsub{\idl{A}} \in \Gamma.
		\]
		As the rest of conditions for $ Z $ are Borel, we have $ Z\in \Gamma $, and its projection $ \idl{I}(U,V) $ is in $ \Gamma $.
		
		This completes the proof that $\mathcal P(\mathbb N)$ is a countable union of filters and ideals from $\Gamma$.		
	\end{proof}
	
	We recall the following well-known results characterizing the ``smallness" of filters.
	
	\begin{lem}[Sierpi\'nski]\label{lemma:smallidl}
		Let $\idl{F}$ be a proper filter. Then it is Lebesgue measurable if and only if it has measure zero, and it has Baire property if and only if it is meager.
	\end{lem}
	\begin{proof}
		See \cite[~Theorem 4.1.1]{bartoszynski1995set}. Note that it is also true for proper ideals.
	\end{proof}
	
	It is well known that all sets being Ramsey implies all filters are meager: it is a theorem of Mathias that $\omega\rtar(\omega)^{\omega}$ implies that all filters on $\omega$ are feeble \cite{mathias1975remark}; also, by M. Talagrand, a filter is feeble if and only if it is meager \cite{talagrand1980compacts}. Below we restate the relativized version of the theorem stated in \cite[~Theorem 7.2]{neeman2018happy}.
	
	\begin{lem}\label{lemma:ramidl}
		If all subsets of $\infsub{\bbn}$ in $\Gamma$ are Ramsey, then all proper filters in $\Gamma$ are meager.
	\end{lem}
	\begin{proof}
		We argue with ideals. If not, let $\idl{I}$ be a non-meager proper ideal in $\Gamma$. So by Lemma \ref{lemma:smallidl} it does not have the Baire property. By \cite[~Theorem 4.1.2]{bartoszynski1995set} we have that
		\[
		\forall x\in \infsub{\bbn} \ \Big[ \exists A \in\idl{I} \ \big[ \text{$\{i\in\bbn:[x(i),x(i+1))\subseteq A\}$ is infinite}\big]\Big]. 
		\]
		Define
		\[
		E(\idl{I})=\left\{x\in \bbn: \bigcup_{i\in\bbn}[x(2i),x(2i+1))\in \idl{I}\right\},
		\]
		which is in $\Gamma$. We show that $E(\idl{I})$ is not Ramsey.
		
		For any $x\in\infsub{\bbn}$, we obtain some $A\in\idl{I}$ as above, and let
		\[
		I = \{i\in\bbn:[x(i),x(i+1))\subseteq A\}.
		\]
		Define $x_1\in\infsub{\bbn}$ such that 
		\[
		\forall n\in\bbn \ \big[ x_1(2n)=x(I(2n)), x_1(2n+1)=x(I(2n)+1) \big].
		\]
		So $x_1\in E(\idl{I})$. Let $ x_2 = x_1\backslash\{x_1(0)\} $. Since
		\[
		\bbn \setminus \lb\bigcup_{i\in\bbn}\big[x_1(2i),x_1(2i+1)\big) \sqcup \bigcup_{i\in\bbn}\big[x_2(2i),x_2(2i+1)\big)\rb \text{ is finite},
		\]
		by properness we have $x_2\notin E(\idl{I})$. This means that for any real $ x\in \infsub{\bbn} $ we have a subreal in $ E(\idl{I}) $ and a subreal not in $ E(\idl{I}) $. So $E(\idl{I})$ is not Ramsey. Thus, $ \idl{I} $ must be meager.
	\end{proof}
	
	We now combine the decomposition theorem with known results on filters and ideals to obtain the main result of this section.
	
	\begin{thm}[$\zfaxiom + \mathsf{CC}_{\mathbb R}$]\label{thm:main-mid}
		Let $\Gamma$ be a good pointclass. If every set in $\Gamma$ is Lebesgue measurable, has Baire property or Ramsey property, then $\Gamma$ contains no MID family.
	\end{thm}
	
	\begin{proof}
		Assume $\Gamma$ contains a MID family $\mathcal A$. By Lemma \ref{lemma:mid-union}, the Cantor space $ \prst(\bbn) $ is a countable union of proper filters and proper ideals in $ \Gamma $.
		\begin{itemize}
			\item[(1)] If all sets in $ \Gamma $ are Lebesgue measurable, then by Lemma \ref{lemma:smallidl} and the countable additivity of measure, $ \prst(\bbn) $ is null, which is impossible.
			\item[(2)] Similarly, if all sets in $ \Gamma $ have the Baire property, then $ \prst(\bbn) $ is meager, contradicting the Baire category theorem.
			\item[(3)] Also, if all sets in $ \Gamma $ have the Ramsey property, then by Lemma \ref{lemma:ramidl} the space $ \prst(\bbn) $ is also meager, which is impossible. 
		\end{itemize}
		Hence $\Gamma$ cannot contain a MID family. Note that $\ccaxiom_\bbr$ is sufficient for all the arguments above.
	\end{proof}
	
	As an immediate application of Theorem \ref{thm:main-mid}, we obtain the following result under the \emph{Axiom of Determinacy} $ \adaxiom $.
	
	\begin{cor} [$\zfaxiom + \adaxiom$]
		There are no MID families.		
	\end{cor}
	
	\begin{proof}
		Under $ \adaxiom $, every set of reals has the Baire property and $ \ccaxiom_\bbr $ holds. Let $\Gamma$ be the pointclass of all subsets of Polish spaces. Theorem \ref{thm:main-mid} implies that $\Gamma$ contains no MID family. Hence no MID family exists.
	\end{proof}

	\section{Ramsey property  and \texorpdfstring{$\mathcal I$}{I}-MAD family}\label{section:imad}
	
	Having established that the Ramsey property eliminates ordinary MAD families, we now extend this result to $ \idl{I} $-MAD families for eventually different ideals and iterated Fubini products of Fr\'echet ideals. We introduce the notion of \emph{dichotomously coded} ideals and show that for any such ideal $\mathcal I$, the Ramsey property precludes the existence of $\mathcal I$-MAD families, provided a sufficiently definable coding operator exists. We then verify that several classes of ideals satisfy this property, including $\mathcal{ED}$, $\mathcal{ED}_{\mathrm{fin}}$, and $\finalpha{\alpha}$ for all countable ordinals $\alpha$.
	
	Motivated by the proof of Theorem \ref{thm:mad-ramsey}, we define the following class of ideals:
	
	\begin{defn}\label{def:dichotomous-coding}
		We say an ideal $\mathcal I \subseteq \mathcal P(\mathbb N)$ is \textbf{dichotomously coded} if for every Dedekind-infinite $\mathcal I$-MAD family $\mathcal A$, there exists a \textbf{coding operator} $\Phi: [\mathbb N]^\omega \to \mathcal I^+$ satisfying the following conditions:
		\begin{itemize}
			\item[(1)] For every $H \in [\mathbb N]^\omega$, there exist $y \in [H]^\omega$ and $A \in \mathcal A$ such that $\Phi(y) \setminus A\in \mathcal I$.
			
			\item[(2)] For every $H \in [\mathbb N]^\omega$, there exist $y \in [H]^\omega$ and $A \in \mathcal A$ such that both $A \cap \Phi(y) \in \mathcal I^+$ and $\Phi(y) \setminus A \in \mathcal I^+$.
		\end{itemize}
	\end{defn}
	
	The following proposition is a straightforward adaptation of Theorem \ref{thm:mad-ramsey}.	
	
	\begin{prop}[$ \zfaxiom $]\label{prop:ramseyimplynoedmad}
		Let $\Gamma$ be a good pointclass and $\mathcal I \subseteq \mathcal P(\mathbb N)$ be a dichotomously coded ideal. If every set in $\Gamma$ has the Ramsey property and there exists a $\Gamma$-measurable coding operator for $\mathcal I$, then $\Gamma$ contains no Dedekind-infinite $\mathcal I$ MAD family.
	\end{prop}
	
	\begin{proof}
		The proof follows exactly the same argument as in Theorem \ref{thm:mad-ramsey}, with the following substitutions:
		\begin{itemize}
			\item[(1)] The ordinary MAD family $\mathcal A$ in Theorem \ref{thm:mad-ramsey} is replaced by the $\mathcal I$-MAD family $\mathcal A$.
			\item[(2)] The operator $\Phi(x)$ defined by $\{A_{x(2n)}(x(2n+1)): n\in\mathbb N\}$ is replaced by the $\Gamma$-measurable coding operator $\Phi$ from Definition \ref{def:dichotomous-coding}.
			\item[(3)] The set $P = \{z \in [\mathbb N]^\omega: \exists A \in \mathcal A, \  \Phi(z)\subseteq A\}$ becomes
			\[
			P = \{z \in [\mathbb N]^\omega: \exists A \in \mathcal A,\ \Phi(z) \setminus A\in \mathcal I\}.
			\]
		\end{itemize}
		The rest of the proof is identical, showing that $P$ is not Ramsey, contradicting the hypothesis that all $\Gamma$-sets are Ramsey.
	\end{proof}

	\subsection{\texorpdfstring{$\mathcal{ED}$}{ED} and \texorpdfstring{$\mathcal{ED}_{\textnormal{fin}}$}{ED\textunderscore{}fin}-MAD Family}
	
	\phantom{a}
	\vspace{6pt}
	
	We now show that the eventually different ideals $\mathcal{ED}$ and $\mathcal{ED}_{\text{fin}}$ belong to the class of dichotomously coded ideals introduced in Definition \ref{def:dichotomous-coding}. This will allow us to apply Proposition \ref{prop:ramseyimplynoedmad} and conclude that the Ramsey property precludes the existence of $\mathcal{ED}$-MAD and $\mathcal{ED}_{\text{fin}}$-MAD families. We first introduce some notation and a special subclass of $\mathcal{ED}^+$.
	
	\begin{defn}
		Define $\mathcal{ED}^{++}$ to be the collection of all $X \in \mathcal{ED}^+$ satisfying the following condition: if $m$ is the $n$-th element of $\mathrm{dom}(X)$, then $|X(m)| = n+1$. We can also define $ \evendiff_{\text{fin}}^{++} \subseteq \evendiff_{\text{fin}}^{+}$ in the same way.
	\end{defn}
	
	For each $X \in \mathcal{ED}^{++}$ and integers $k \le n$, we denote
	\[
	X [ n,k ] = (m_n, m_k),
	\]
	where $m_n$ is the $n$-th element of $\mathrm{dom}(X)$ and $m_k$ is the $k$-th element of $X(m_n)$. We denote $ m_n $ by $ X[n] $.
	
	We need the following fact.
	
	\begin{fact}
		For any $ X \in \evendiff^+$, there is a subset $ Y\subseteq X $ such that $ Y\in\evendiff^{++} $. This also holds for $ \evendiff_{\textnormal{fin}} $.
	\end{fact}
	
	\begin{prop}\label{prop:eddichotomous}$\mathcal{ED}$ and $\mathcal{ED}_{\textnormal{fin}}$ are dichotomously coded.	
	\end{prop}
	
	\begin{proof}
		We give the proof for $\mathcal{ED}$; the case of $\mathcal{ED}_{\text{fin}}$ is analogous, using $\mathcal{ED}_{\text{fin}}^{++}$ in place of $\mathcal{ED}^{++}$.
		
		Let $\mathcal A$ be a Dedekind-infinite $\mathcal{ED}$-MAD family. Fix a sequence $\{A_n: n\in \bbn\}$ of distinct elements in $\mathcal A$. We may assume that the domains $\{\mathrm{dom}(A_n): n\in \bbn\}$ are pairwise disjoint and that each $A_n$ belongs to $\mathcal{ED}^{++}$:
		\begin{itemize}
			\item[(1)] We shrink every $ A_n $ so that $ A_n\in \evendiff^{++} $. Note that $ \dom(A_n) $ is infinite.
			\item[(2)] Define pairwise disjoint infinite subsets $ Q_n = \{q_{n,j}:j\in\bbn\} \subseteq \dom(A_n)$ in the following order:
			\[
			q_{0,0}< q_{1,0}< q_{1,1}<q_{0,1}< q_{2,0}<\cdots.
			\]
			This is done by picking $ q_{n,j} \in \dom(A_n)$ such that it is larger than any previously picked $ q $'s. 
			\item[(3)] Replace $ A_n $ by $ \{(j,k)\in A_n: j\in Q_n\} $ and further shrink to $ \evendiff^{++} $ sets.
		\end{itemize}
		Note that for any $ n\in\bbn  $ we have $ A_n[0] < A_{n+1}[0] $.
		
		Define the coding operator $\Phi: [\mathbb N]^\omega \to \mathcal{ED}^+$ as
		\[
		\Phi(x) = \bigcup_{n \in \mathbb N} \bigcup_{i=0}^{n} A_{x\left( \frac{n(n+5)}{2} \right)} \Big[ x\left( \tfrac{(n+1)(n+6)}{2} - 1 \right), x\left( \tfrac{n(n+5)}{2} + 1 + i \right) \Big]. \tag{\dag}
		\]
		
		The definition is illustrated in Figure \ref{fig:tildeed}.
		
		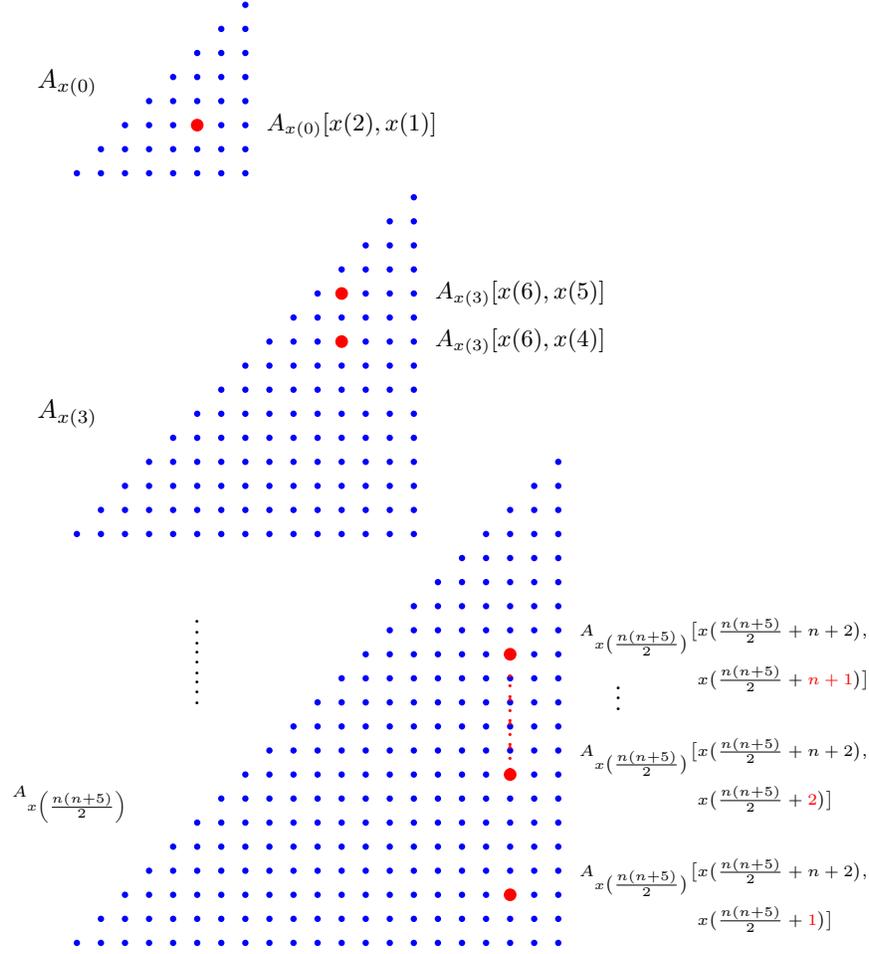
\begin{figure}[h]
			\centering
			\begin{tikzpicture}[scale=0.8]

				\foreach \col in {0,1,2,3,4,5,6,7} {
					\foreach \row in {0,...,\col} {
						\fill[blue] (\col*0.4, \row*0.4) circle(1.5pt);
					}
				}

				\node[left] at (0.5, 1.5) {$A_{x(0)}$};

				\foreach \col in {0,1,2,3,4,5,6,7,8,9,10,11,12,13,14} 	{
					\foreach \row in {0,...,\col} {
						\fill[blue] (\col*0.4, \row*0.4 - 6) circle(1.5pt);
					}
				}

				\node[left] at (0.5, -4.0) {$A_{x(3)}$};

				\fill[red] (5*0.4, 2*0.4) circle(3pt);
				\node[right] at (5*0.4 + 1, 2*0.4) {\small $ A_{x(0)}[x(2),x(1)] $};

				\fill[red] (11*0.4, 8*0.4 - 6) circle(3pt);
				\node[right] at (12*0.4 + 1, 8*0.4 - 6) {\small $ A_{x(3)}[x(6),x(4)] $};
				
				\fill[red] (11*0.4, 10*0.4 - 6) circle(3pt);
				\node[right] at (12*0.4 + 1, 10*0.4 - 6) {\small $ A_{x(3)}[x(6),x(5)] $};

				\node at (2, -7.5) {$\vdots$};
				\node at (2, -8.0) {$\vdots$};
				\node at (2, -8.5) {$\vdots$};

				\foreach \col in {0,1,2,3,4,5,6,7,8,9,10,11,12,13,14,15,16,17,18,19,20} {
					\foreach \row in {0,...,\col} {
						\fill[blue] (\col*0.4, \row*0.4 - 12.8) circle(1.5pt);
					}
				}

				\node[left] at (1, -10.5) {\tiny$A_{x\left(\frac{n(n+5)}{2}\right)}$};

				\fill[red] (18*0.4, 10*0.4 - 12) circle(3pt);
				\node[red, font=\Large] at (18*0.4, 9*0.4 - 12){$\vdots$};
				\node[red, font=\Large] at (18*0.4, 8*0.4 - 12){$\vdots$};
				\node[red, font=\Large] at (18*0.4, 7*0.4 - 12){$\vdots$};
				\node[red, font=\Large] at (18*0.4, 6*0.4 - 12){$\vdots$};
				\fill[red] (18*0.4, 5*0.4 - 12) circle(3pt);
				\fill[red] (18*0.4, 0*0.4 - 12) circle(3pt);
				
				\node[right] at (18*0.4 + 1, 10*0.4 - 12) {\tiny$\begin{aligned}
						A_{x\big(\tfrac{n(n+5)}{2}\big)}\big[&x\big(\tfrac{n(n+5)}{2}+n+2\big),\\
						&x\big(\tfrac{n(n+5)}{2}+\textcolor{red}{n+1}\big)\big]
					\end{aligned}$};
				
				\node at (20*0.4 + 1, 8.5*0.4 - 12) {$\vdots$};
				
				\node[right] at (18*0.4 + 1, 5*0.4 - 12) {\tiny$\begin{aligned}
						A_{x\big(\tfrac{n(n+5)}{2}\big)}\big[&x\big(\tfrac{n(n+5)}{2}+n+2\big),\\
						&x\big(\tfrac{n(n+5)}{2}+\textcolor{red}{2}\big)\big] 
					\end{aligned}$};
				
				\node[right] at (18*0.4 + 1, 0*0.4 - 12) {\tiny$\begin{aligned}
						A_{x\big(\tfrac{n(n+5)}{2}\big)}\big[&x\big(\tfrac{n(n+5)}{2}+n+2\big),\\
						&x\big(\tfrac{n(n+5)}{2}+\textcolor{red}{1}\big)\big] 
					\end{aligned}$};	
			\end{tikzpicture}
			\caption{Definition of $\Phi(x)$ for $x=\{0,2,5,7,8,10,11,\cdots \}$.}
			\label{fig:tildeed}
		\end{figure}
		
		Note that for any $ x\in \infsub{\bbn} $ we have $ \Phi(x)\in\evendiff^+ $. Now we verify that $\Phi$ satisfies the two conditions in Definition \ref{def:dichotomous-coding}.
		
		\begin{claim}{\label{claim:edhide}}
			For every $H \in [\mathbb N]^\omega$, there exist $y \in [H]^\omega$ and $A \in \mathcal A$ such that $\Phi(y) \subseteq  A$.
		\end{claim}
		
		\begin{claimproof}{\ref{claim:edhide}}
			Since $\mathcal A$ is an $\mathcal{ED}$-MAD family and $\Phi(H) \in \mathcal{ED}^+$, by maximality there exists $A \in \mathcal A$ such that $\Phi(H) \cap A \in \mathcal{ED}^+$.
			
			We construct $y \subseteq H$ recursively. For convenience, let $ y(-1) = -1 $. Assume we have already defined $y(j)$ for $j < \frac{k(k+5)}{2}$. We now define the next block of elements corresponding to index $k$, i.e. from $ y\left( \tfrac{k(k+5)}{2} \right) $ to $ y\left( \tfrac{(k+1)(k+6)}{2} - 1 \right) $.
			
			By the definition of $\Phi(H)$, for the fixed $H\left(\frac{(n+1)(n+6)}{2}-1\right)$-th column of $ A_{H\left( \frac{n(n+5)}{2} \right)} $, the vertical indices appearing in $\Phi(H)$ are exactly the values $H(m)$ for $m$ in the interval $\left[\frac{n(n+5)}{2}+1, \frac{(n+1)(n+6)}{2}-1\right)$.	
			Since $\Phi(H) \cap A \in \mathcal{ED}^+$, there are infinitely many such columns for which this intersection contains at least $k+1$ elements. Thus for each $k$ we can find minimal $n_k$ such that
			\[
			H\left( \tfrac{n_k(n_k+5)}{2} \right)>y\left( \tfrac{k(k+5)}{2} -1\right)
			\]
			and the interval
			\[
			\left[ \frac{n_k(n_k+5)}{2}+1,\ \frac{(n_k+1)(n_k+6)}{2} -1\right)
			\]
			contains $k+1$ indices $j^k_0 < \cdots < j^k_k$ satisfying
			\[
			A_{H\left( \frac{n_k(n_k+5)}{2} \right)}\Big[ H\left( \tfrac{(n_k+1)(n_k+6)}{2} - 1 \right), H(j^k_i) \Big] \in A, \text{ for all } i \le k.
			\]
			
			Define $y$ by setting
			\[
			\begin{aligned}
				y\left( \tfrac{k(k+5)}{2} \right) &= H\left( \tfrac{n_k(n_k+5)}{2} \right), \\
				y\left( \tfrac{k(k+5)}{2} + 1 + i \right) &= H(j^k_i), \ 0 \le i \le k, \\
				y\left( \tfrac{(k+1)(k+6)}{2} - 1 \right) &= H\left( \tfrac{(n_k+1)(n_k+6)}{2} - 1 \right).
			\end{aligned}
			\]
			
			By construction, every term in $\Phi(y)$ belongs to $A$, hence $\Phi(y) \subseteq A$.
		\end{claimproof}
		
		\begin{claim}{\label{claim:edsplit}}
			For every $H \in [\mathbb N]^\omega$, there exist $y \in [H]^\omega$ and $A \in \mathcal A$ such that both $A \cap \Phi(y)$ and $\Phi(y) \setminus A$ belong to $\mathcal{ED}^+$.
		\end{claim}
		
		\begin{claimproof}{\ref{claim:edsplit}}
			Fix $H \in [\mathbb N]^\omega$. By Claim \ref{claim:edhide}, pick $A \in \mathcal A$ and $x \in [H]^\omega$ such that $\Phi(x) \subseteq A$.
			
			Note that $A \neq A_n$ for all $n \in \bbn$, so we have $A \cap A_n \in \mathcal{ED}$ for every $n\in\bbn$. 
			This induces an increasing sequence $\{l_n : n \in \bbn\}$ such that for each $n\in\bbn$,
			\[
			\forall m > l_n [ |(A_n \cap A)(m)| < l_n].
			\]
			
			We construct $y \in [x]^\omega$ recursively. Let $ y(-1) = -1 $ and  for each $k \in \bbn$, assume we have already defined $y(j)$ for $j < \frac{k(k+5)}{2}$. We now define $ y\left( \tfrac{k(k+5)}{2} \right),\cdots, y\left( \tfrac{(k+1)(k+6)}{2} - 1 \right) $.
			
			\begin{itemize}
				\item[(1)] When $ k $ is even, we do exactly as in Claim \ref{claim:edhide}.
				\item[(2)]  When $ k $ is odd,
				\begin{itemize}
					\item[i.] Let
					\[
					n_k = \min\left\{n\in \bbn: x\lb \tfrac{n(n+5)}{2} \rb > y\left( \tfrac{k(k+5)}{2}-1 \right)\right\}
					\]
					and define $ y\left( \tfrac{k(k+5)}{2} \right) = x\lb \tfrac{n_k(n_k+5)}{2} \rb $.
					\item[ii.] Let
					\[
					\begin{aligned}
						y\left( \tfrac{(k+1)(k+6)}{2} - 1 \right) = \min\big\{m\in x:\ &m>y_k \ \wedge \ A_{y_k}[m] > l_{y_k}\\
						\wedge \ &\left| \left( y_k, m\right) \cap x  \right|> k+1+ l_{y_k}\big\},
					\end{aligned}
					\]
					where $ y_k = y\left(\tfrac{k(k+5)}{2}\right) $ and $ \left( y_k, m\right) $ denotes an open interval.
					\item[iii.] Denote $ y_{k+1} =y\left( \tfrac{(k+1)(k+6)}{2} - 1 \right) $. By the definition above, we can find
					\[
					y_k < q_0^k < \cdots < q_k^k < y_{k+1}
					\]
					such that for any $ i \leq k, q_i^k \in x $ and
					\[
					A_{ y\left(\tfrac{k(k+5)}{2}\right)}\left[ y\left( \tfrac{(k+1)(k+6)}{2} - 1 \right), q_i^k \right]\notin A, \forall i\leq k.
					\]
					Such $ q_i^k $'s exist, because at this $ y_{k+1} $-th column of $ A_{y_k} $, there are less than $ l_{y_k} $-many indices in $ (y_k, y_{k+1})\cap x $ which substitute into $ A $, and there are more than $ \lb k+1+l_{y_k}\rb $-many points in total. Define
					\[
					y\left(\tfrac{k(k+5)}{2}+i+1\right) = q_i^k, \forall i\leq k.
					\]
				\end{itemize}
			\end{itemize}
			
			By construction, for each $k$, the $k+1$ terms in $\Phi(y)$ coming from index $k$ are all inside $A$ when $k$ is even, and all outside $A$ when $k$ is odd. Hence  both $A \cap \Phi(y)$ and $\Phi(y) \setminus A$  are in $\mathcal{ED}^+$.
		\end{claimproof}
		
		Since $\Phi$ satisfies both conditions of Definition \ref{def:dichotomous-coding}, we conclude that $\mathcal{ED}$ is a dichotomously coded ideal.	
	\end{proof}
	
	We now combine the results of this section to obtain the main theorem on $\mathcal{ED}$ and $\mathcal{ED}_{\text{fin}}$-MAD families.
	
	\begin{thm}[$ \zfaxiom $]\label{theroem:ramseyimplynoed}
		Let $\Gamma$ be a good pointclass. If every set in $\Gamma$ has Ramsey property, then there is no Dedekind-infinite $\mathcal{ED}$-MAD family or Dedekind-infinite $\mathcal{ED}_{\textnormal{fin}}$-MAD family in $\Gamma$.	
	\end{thm}
	
	\begin{proof}
		Note that the coding operator constructed in Proposition \ref{prop:eddichotomous} is continuous. 
		Applying Proposition \ref{prop:ramseyimplynoedmad} together with the assumption that every set in $\Gamma$ has the Ramsey property yields the desired conclusion.
	\end{proof}
	
	As an immediate consequence, we obtain the following result.
	
	\begin{cor}
		There is no analytic infinite $\mathcal{ED}$-MAD family or analytic infinite $\mathcal{ED}_{\textnormal{fin}}$-MAD family.
	\end{cor}	
	
	\begin{proof}
		Let $ \Gamma $ be the pointclass of analytic subsets of  $[\mathbb N]^\omega$ and apply Theorem \ref{theroem:ramseyimplynoed}.	
	\end{proof}

	\subsection{\texorpdfstring{$ \finalpha{\alpha} $}{Fin\textasciicircum{} α}-MAD Family}
	
	\phantom{a}
	\vspace{6pt}

	We now turn to the ideals $\finalpha{\alpha}$ for countable ordinals $\alpha$. The proof that these ideals are dichotomously coded follows the same pattern as for $\mathcal{ED}$, but requires a careful handling of the tree structure $S_\alpha$. To this end, we first introduce the analog of $\mathcal{ED}^{++}$ for the $\finalpha{\alpha}$ hierarchy. We recall Section \ref{section:prelim} for the notations used below.
	
	\begin{defn}[\cite{schrittesser2020ramsey}]
		For each countable ordinal $\alpha$ with $1 \le \alpha < \omega_1$, we recursively define the collection $\finalpha{\alpha++} \subseteq \mathcal P(S_\alpha)$ as follows:
		\begin{itemize}
			\item[(1)] If $\alpha = 1$, then $\finalpha{1++} = \finalpha{1+}$.
			\item[(2)] If $\alpha \ge 2$, then $X \subseteq S_\alpha$ belongs to $\finalpha{\alpha++}$ if and only if $\mathrm{dom}(X)$ is infinite and for every $n \in \mathrm{dom}(X)$, we have $X(n) \in \finalpha{\gamma^\alpha_n++}$.
		\end{itemize}	
	\end{defn}
	
	We list some useful facts here.
	
	\begin{fact}{\label{fact:plusplus}}
		Let $ \alpha \geq 1 $.
		\begin{itemize}
			\item[(1)] Every $X \in \mathrm{Fin}^{\alpha+}$ contains a subset $Y \subseteq X$ with $Y \in \mathrm{Fin}^{\alpha++}$.
			\item[(2)] For $X \in \mathrm{Fin}^{\alpha++}$, we have $\rho_X(\emptyset) = \alpha$, and for any $t \in T(X)$ with $\rho_X(t) > 0$, we have $X(t) \in \mathrm{Fin}^{\rho_X(t)++}$.
			\item[(3)] A set $ X $ is in $\mathrm{Fin}^{\alpha++} $ if and only if it is non-empty and for any $ s\in T(X) $ with $ \rho_X(s)\geq 1 $, we have $ \mathrm{dom}(X(s)) $ is infinite.
		\end{itemize}
	\end{fact}
	
	\begin{proof}
		For brevity, we outline the proofs as follows.
		\begin{itemize}
			\item[(1)] This is proved via induction on $ \alpha $ and the proof is omitted.
			\item[(2)]
			\begin{itemize}
				\item[i.] Note that, via a simple induction, for any set $X \subseteq S_\alpha$, integer $n$, and sequence $m \in T(X(n))$, we have $ \rho_X(n^\frown m) = \rho_{X(n)}(m) $. From this we can prove the first part.
				\item[ii.] For the second part, note that $ X(t)\in \finalpha{\rho_{X(t)}(\emptyset)++} $ and we have $\rho_{X(t)}(\emptyset)  = \rho_X(t) $ by iteratively using $ \rho_X(n^\frown m) = \rho_{X(n)}(m) $.
			\end{itemize}
			\item[(3)] Assume $ X\in \finalpha{\alpha++} $, then it is non-empty, and by (2) the condition on the right-hand side holds. Conversely, we do induction on $ \alpha $.
			\begin{itemize}
				\item[i.] For $ \alpha=1 $, we have $ T(X) = \{\emptyset\}\cup X $ so $ \rho_X(\emptyset)=1 $. By the assumption, $ X=X(\emptyset) $ is infinite.
				\item[ii.] For $ \alpha>1 $, as $ X \not=\emptyset$, we have $ \rho_X(\emptyset)>1 $. So by the assumption $ \dom(X) = \dom(X(\emptyset)) $ is infinite. Also, for any $ n\in \dom(X) $, by definition $ X(n)\not=\emptyset $. Since for any $ s\in T(X(n)) $ we have $ \rho_{X(n)}(s) = \rho_X(n^\frown s) $ and $ \dom (X(n)(s))  = \dom(X(n^\frown s))$, $ X(n) $ has satisfied the assumption. By the induction hypothesis we have $ X(n)\in \finalpha{\gamma_n^\alpha ++} $ and thus $ X\in \finalpha{\alpha++} $.
			\end{itemize}
		\end{itemize}
	\end{proof}
	
	Recall that for any $ x\subseteq \bbn $, its $ n $-th element is denoted by $ x(n) $. Here we define its higher dimensional analog.
	
	\begin{defn}
		For $ \alpha \geq 1 $, $X \subseteq S_\alpha$, $ s\in T(X) $ with $ |s|>0 $, and $t \in \finfunc{\bbn}$, 
		\begin{itemize}
			\item[(1)] We say that $ t $ is an \textbf{$ X $-index} of $ s $, and $ s $ is an \textbf{$ X $-evaluation} of $ t $, if $ |s| = |t| $ and for any $ i\in [0,|s|)$, we have that $ s(i) $ is the $ t(i) $-th element of $ \dom(X(s\uphar i)) $. Let
			\[
			\indexsequence(X)=\bigg\{t \in \finfunc{\bbn}: \exists s\in T(X) \  \big[t \text{ is an } X\text{-index of } s\big]\bigg\}
			\]
			as the set of indices of $ X $. Note that for any $ t\in \indexsequence(X) $ it corresponds to a unique $ s \in T(X)$, which is denoted by $ X[ t ] $, and we denote $ X(X[t]) $ by $ X\llb t \rrb $. Also, for any $ s\in T(X) $ there is a unique $ X $-index, which is denoted by $ X^{-1} [ s ] $.
			\item[(2)] 	For $X \subseteq S_\alpha$ and $t \in \finfunc{\bbn}$, we say $t$ is a \textbf{terminal sequence} of $X$ if $ t\in\indexsequence(X) $ and $ X[t] \in X $. Let $ \terminalsequence(X) $ be the set of all terminal sequences of $ X $.
		\end{itemize}
	\end{defn}
	
	Here we note that $ \dom(X(s\uphar 0)) = \dom(X) $, and for any $ s\in X $ with $ |s|\geq 1 $, we have $ \dom(X(s\uphar (|s|-1))) = X(s\uphar (|s|-1)) $. For convenience, we say the empty sequence is the index of itself. By our definition, no proper initial segment of a terminal sequence is itself a terminal sequence; each terminal sequence corresponds to a maximal finite path in $T(X)$. 
	
	The following proposition is also useful.
	
	\begin{prop}{\label{prop:plusplusterminal}}
		Let $ \alpha\geq 1 $. For any $ A\subseteq S_\alpha $, we have that $ A\in\finalpha{\alpha++} $ if and only if for any $ t\in\inffunc{\bbn} $, there is unique $ i\geq 1$ such that $ t\uphar i \in \terminalsequence(A)$. In this case we denote $ t\uphar i $ by $ t_A $.
	\end{prop}
	
	\begin{proof}
		\phantom{a}
		\begin{itemize}
			\item[(1)] 	Suppose $ A\in\finalpha{\alpha++} $. We show that for any $ t\in\inffunc{\bbn} $, there is $ i\geq 1$ such that $ t\uphar i \in \terminalsequence(A)$. Such $ i $ is unique if it exists. We proceed by induction on $ \alpha $.
			\begin{itemize}
				\item[i.] For $ \alpha = 1 $, the statement is trivial: for any $ t $, take $ i=1 $.
				\item[ii.] For $ \alpha >1$, suppose this direction holds for any $ \beta<\alpha $. Fix $ t\in \inffunc{\bbn} $. As $  A\in\finalpha{\alpha++} $, we have $ \dom(A) $ is infinite, so the $ t(0) $-th element $ A[t(0)] $ exists. Also, $ A\llb t(0) \rrb \in \finalpha{\gamma(\alpha, A[t(0)])++} $, thus by the induction hypothesis applied to $ t\uphar [1,+\infty) $ we obtain $ i $ such that $ (t\uphar [1,+\infty))\uphar i\in \terminalsequence(A\llb t(0) \rrb) $. This means $ t\uphar (i+1)\in \terminalsequence(A) $. 
			\end{itemize}
			\item[(2)] Conversely, let $ A\subseteq S_\alpha $ satisfy the condition on the right-hand side. We also proceed by induction on $ \alpha $.
			\begin{itemize}
				\item[i.] For $ \alpha = 1 $, by the assumption, for any $ t\in \inffunc{\bbn} $, there exists $ i\geq 1 $ such that $ A[t\uphar i]\in A $. By the definition of the terminal sequence, we have $ A[t\uphar i](0) $ is the $ t(0) $-th element of $ A $, which holds for every $ t(0)\in \bbn $. Thus $ A\in \finalpha{1++} $.
				\item[ii.] For $ \alpha >1 $, suppose this direction holds for any $ \beta<\alpha $. Similarly we have that $A[t\uphar i](0) $ is the $ t(0) $-th element of $ \dom(A) $, which holds for every $ t(0)\in \bbn $. Thus $ \dom(A) $ is infinite. We show that for any $ n\in\dom(A) $ and for any $ t\in\inffunc{\bbn} $ there is $ i\geq 1  $ with $ t\uphar i\in \terminalsequence(A(n)) $, and by the induction hypothesis $ A(n)\in \finalpha{\gamma_n^\alpha ++} $. For any $  n\in \dom(A) $, let $ n=A[k] $. By assumption on $ A $, for any such $ k\in\bbn $, there is $ j $ such that $ A[(k^\frown t)\uphar j] \in A\subseteq S_\alpha$. As $ \alpha\geq 2 $ we have $ j\geq 2 $. Hence $ t\uphar (j-1)\in \terminalsequence(A(n)) $. Thus $ A $ satisfies the definition of $ \finalpha{\alpha++} $ sets, which completes the proof.
			\end{itemize}
		\end{itemize}
	\end{proof}
	
	We now extend the dichotomous-coding property to the $\finalpha{\alpha}$ hierarchy. The key idea remains the same: from a given $\finalpha{\alpha}$-MAD family, we extract a suitable sequence $\{A_n: n\in\bbn\}$ with pairwise disjoint domains and $A_n \in \finalpha{\alpha++}$; then we define a coding operator $\Phi$ using a similar indexing scheme, and finally we verify the two dichotomous-coding conditions via recursive constructions that parallel those in Proposition \ref{prop:eddichotomous}. Regarding the indexing scheme for $\finalpha{\alpha}$, we need to encode arbitrary finite sequences of natural numbers to navigate through the levels of the tree. This is achieved by fixing a bijection $\pi$ that maps finite sequences to natural numbers. Especially, fix a bijection 
	\[
	\pi: \mathbb N^{<\omega}\setminus\left(\{\emptyset\}\cup \bbn\right) \to \mathbb N
	\]
	such that:
	\begin{itemize}
		\item[(1)] If $s \subsetneqq t$, then $\pi(s) < \pi(t)$;
		\item[(2)] If $s \in \mathbb N^{<\omega}$ and $n < m$, then $\pi(s^\frown n) < \pi(s^\frown m)$ and $\pi(n^\frown s) < \pi(m^\frown s)$.
	\end{itemize}
	Note that for any $ n\in \bbn $ we have $ n\leq \pi(n,0) $. For convenience, we set $ \pi(n) = -1 $ for any $ n\in\bbn $. Such a coding function for finite sequences is standard in the literature \cite[Section~1C]{moschovakis2025descriptive}. For instance, let $ p_i $ denote the $ i $-th prime with $ p_0 = 2 $, and define
	\[
	c(s) = \prod_{i<|s|} p_i^{s(i)+1},
	\]
	which is injective. Let $ f $ be the enumeration of the image $ c''\left( \mathbb N^{<\omega}\setminus\left(\{\emptyset\}\cup \bbn\right)\right) $, and then set $ \pi = f\circ c $.

	\begin{prop}\label{prop:findichotomous}
		For each countable ordinal $\alpha$ with $1 \le \alpha < \omega_1$, the ideal $\finalpha{\alpha}$ is dichotomously coded.	
	\end{prop}

	\begin{proof}
		For $ \alpha = 1 $, it is already proved in Theorem \ref{thm:mad-ramsey}. Thus we assume $ \alpha >1 $ and let $\mathcal A$ be a Dedekind-infinite $\mathrm{Fin}^\alpha$-MAD family. Fix a sequence $\{A_n\}_{n\in\mathbb N}$ of distinct elements from $\mathcal A$. As in Proposition \ref{prop:eddichotomous}, we may assume that the domains $\{\mathrm{dom}(A_n) : n \in \mathbb N\}$ are pairwise disjoint and that each $A_n$ belongs to $\finalpha{\alpha++}$:
		\begin{itemize}
			\item[(1)] For any $ A_n $, define $ P_n = \{m\in\bbn: A_n(m)\in \finalpha{\gamma_m^\alpha +}\} $, which is infinite.
			\item[(2)] Define pairwise disjoint infinite subsets $ Q_n = \{q_{n,j}:j\in\bbn\} \subseteq P_n$ using the square spiral order as in Proposition \ref{prop:eddichotomous}.
			\item[(3)] Replace $ A_n $ by $ \{s\in A_n: s(0)\in Q_n\} $ and further shrink these sets to $ \finalpha{\alpha++} $ sets.
		\end{itemize}
		Also note that $ A_n[0] < A_{n+1}[0] $ for every $ n\in\bbn $.
		
		For every $x \in [\mathbb N]^\omega$, $ n\in\bbn $ and $ t\in \finfunc{\bbn} \cup \inffunc{\bbn}$ with $ |t| \geq 1$, define the sequence $ m(x,n,t) $ by
		\[
		m(x,n,t)(i) = x(3\pi(n, t \uphar (i+1)) + 2), \text{ for any }  i\in [0, |t|),
		\]
		and let $ m(x,n,\emptyset) = \emptyset $ for convenience. When $ x $ is finite, we also define $ m(x,n,t) $ as above whenever $ t $ is not beyond the range of $ x $.
		
		Denote
		\[A_{x(3n)}[x(3n+1)]=k_{x,n} , \  A_{x(3n)}\llb x(3n+1)\rrb=K_{x,n}.
		\]
		Now, the coding operator $\Phi(x) \subseteq S_\alpha$ is defined by
		\[
		\Phi(x) = \bigcup_{n \in \mathbb{N}} \bigcup_{t \in \inffunc{\bbn}}{k_{x,n}}^\frown K_{x,n}[ m(x,n,t)_{K_{x,n}}], \tag{\ddag}
		\]
		where $ m(x,n,t)_{K_{x,n}} $ follows the notation in Proposition \ref{prop:plusplusterminal}. Note that the second union is actually a countable union.
		
		Figure \ref{fig:tildefinalpha} is a possible illustration of $ \Phi(x) $ at $ A_{x(3n)} $ for a single index, where we assume $ \alpha = \omega $, $ \gamma_n^\omega = n+1 $ and $ \dom(A_{x(3n)}) = \bbn $. As $ A_{x(3n)} $ is a $ {}^{++} $-set, blue nodes illustrate that it is everywhere infinitely branching. Red nodes represent the maximal path chosen by the index $ x(3n+1)^\frown m(x,n,t) $.
		
		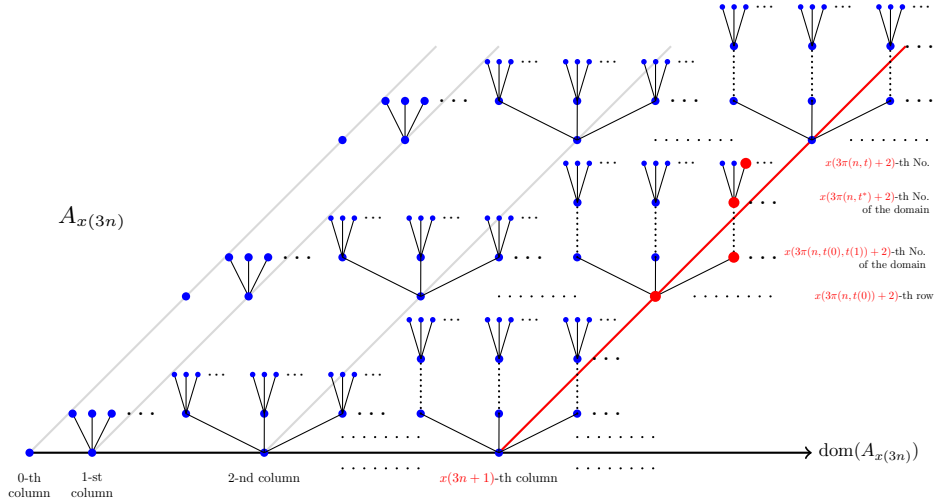
\begin{figure}[h]
			\centering
			\begin{tikzpicture}[scale=1.04]
				\node at (2.8,3) {\small$A_{x(3n)}$};
				
				\foreach \x in {2, 2.8, 5, 8} {
					\draw[gray!30, thick] (\x, 0) -- (\x+5.2, 5.2);
				}
				
				\draw[red, thick] (8, 0) -- (8+5.2, 5.2);
				
				\draw[->, thick] (2,0) -- (12, 0) node[right, scale=0.7] {$\dom(A_{x(3n)})$};

				\newcommand{\treeOne}[2]{
					\draw (#1,#2) -- ++(-0.25, 0.5);
					\draw (#1,#2) -- ++(0, 0.5);
					\draw (#1,#2) -- ++(0.25, 0.5);
					\fill[blue] (#1,#2) circle (1.5pt);
					\fill[blue] (#1-0.25,#2+0.5) circle (1.5pt);
					\fill[blue] (#1,#2+0.5) circle (1.5pt);
					\fill[blue] (#1+0.25,#2+0.5) circle (1.5pt); 
					\node at (#1+0.25+0.4, #2+0.5) {$\dots$};
				}

				\newcommand{\treeTwo}[2]{
					\coordinate (L) at (#1-1, #2+0.5);
					\coordinate (M) at (#1, #2+0.5);
					\coordinate (R) at (#1+1, #2+0.5);
					\fill[blue] (#1,#2) circle (1.5pt);
					\fill[blue] (#1-1, #2+0.5) circle (1.5pt);
					\fill[blue] (#1,#2+0.5) circle (1.5pt);
					\fill[blue] (#1+1,#2+0.5) circle (1.5pt);
					
					\draw (#1,#2) -- (L); 
					\draw (#1,#2) -- (M); 
					\draw (#1,#2) -- (R);
					\node at (#1+1+0.4, #2+0.5) {$\dots$};
					
					\draw (L) -- ++(-0.15, 0.5) coordinate (treetwoLL);
					\fill[blue] (treetwoLL) circle (1pt);
					\draw (L) -- ++(0, 0.5) coordinate (treetwoLM);
					\fill[blue] (treetwoLM) circle (1pt);
					\draw (L) -- ++(0.15, 0.5) coordinate (treetwoLR);
					\fill[blue] (treetwoLR) circle (1pt);
					
					\node[xshift=0.25cm] at (treetwoLR) {\scalebox{0.5}{${\bm\cdots}$}};

					\draw (M) -- ++(-0.15, 0.5) coordinate (treetwoML);
					\fill[blue] (treetwoML) circle (1pt);
					\draw (M) -- ++(0, 0.5) coordinate (treetwoMM);
					\fill[blue] (treetwoMM) circle (1pt);
					\draw (M) -- ++(0.15, 0.5) coordinate (treetwoMR);
					\fill[blue] (treetwoMR) circle (1pt);
					
					\node[xshift=0.25cm] at (treetwoMR) {\scalebox{0.5}{${\bm\cdots}$}};
					
					\draw (R) -- ++(-0.15, 0.5) coordinate (treetwoRL);
					\fill[blue] (treetwoRL) circle (1pt);
					\draw (R) -- ++(0, 0.5) coordinate (treetwoRM);
					\fill[blue] (treetwoRM) circle (1pt);
					\draw (R) -- ++(0.15, 0.5) coordinate (treetwoRR);
					\fill[blue] (treetwoRR) circle (1pt);
					
					\node[xshift=0.25cm] at (treetwoRR) {\scalebox{0.5}{${\bm\cdots}$}};
				}

				\newcommand{\treeN}[3]{
					\coordinate (L#3) at (#1-1, #2+0.5);
					\coordinate (M#3) at (#1, #2+0.5);
					\coordinate (R#3) at (#1+1, #2+0.5);
					\fill[blue] (#1,#2) circle (1.5pt);
					\fill[blue] (#1-1, #2+0.5) circle (1.5pt);
					\fill[blue] (#1,#2+0.5) circle (1.5pt);
					\fill[blue] (#1+1,#2+0.5) circle (1.5pt);
					
					\draw (#1,#2) -- (L#3); 
					\draw (#1,#2) -- (M#3); 
					\draw (#1,#2) -- (R#3);
					\node at (#1+1+0.4, #2+0.5) {$\dots$};
					
					\draw[dash pattern=on 0pt off 2pt, line cap=round, thick] (L#3) -- ++(0, 0.7) coordinate (L2#3);
					\fill[blue] (L2#3) circle (1.5pt);
					
					\draw[dash pattern=on 0pt off 2pt, line cap=round, thick] (M#3) -- ++(0, 0.7) coordinate (M2#3);
					\fill[blue] (M2#3) circle (1.5pt);
					
					\draw[dash pattern=on 0pt off 2pt, line cap=round, thick] (R#3) -- ++(0, 0.7) coordinate (R2#3);
					\fill[blue] (R2#3) circle (1.5pt);
					
					\node at (#1+1+0.4, #2+0.5+0.7) {$\dots$};

					\draw (L2#3) -- ++(-0.15, 0.5) coordinate (treeNLL#3);
					\fill[blue] (treeNLL#3) circle (1pt);
					\draw (L2#3) -- ++(0, 0.5) coordinate (treeNLM#3);
					\fill[blue] (treeNLM#3) circle (1pt);
					\draw (L2#3) -- ++(0.15, 0.5) coordinate (treeNLR#3);
					\fill[blue] (treeNLR#3) circle (1pt);
					
					\node[xshift=0.25cm] at (treeNLR#3) {\scalebox{0.5}{${\bm\cdots}$}};

					\draw (M2#3) -- ++(-0.15, 0.5) coordinate (treeNML#3);
					\fill[blue] (treeNML#3) circle (1pt);
					\draw (M2#3) -- ++(0, 0.5) coordinate (treeNMM#3);
					\fill[blue] (treeNMM#3) circle (1pt);
					\draw (M2#3) -- ++(0.15, 0.5) coordinate (treeNMR#3);
					\fill[blue] (treeNMR#3) circle (1pt);
					
					\node[xshift=0.25cm] at (treeNMR#3) {\scalebox{0.5}{${\bm\cdots}$}};
					
					\draw (R2#3) -- ++(-0.15, 0.5) coordinate (treeNRL#3);
					\fill[blue] (treeNRL#3) circle (1pt);
					\draw (R2#3) -- ++(0, 0.5) coordinate (treeNRM#3);
					\fill[blue] (treeNRM#3) circle (1pt);
					\draw (R2#3) -- ++(0.15, 0.5) coordinate (treeNRR#3);
					\fill[blue] (treeNRR#3) circle (1pt);
					
					\node[xshift=0.25cm] at (treeNRR#3) {\scalebox{0.5}{${\bm\cdots}$}};
				}

				\node[below=0.2cm,align=center,scale=0.5] at (2,0) {$ 0 $-th\\column};
				\node[below=0.2cm,align=center,scale=0.5] at (2.8,0) {$ 1 $-st\\column};
				\node[below=0.2cm,scale=0.5] at (5,0) {$ 2 $-nd column};
				\draw[dash pattern=on 0pt off 4pt, line cap=round, thick] (6, -0.2) -- (7, -0.2);
				\node[below=0.2cm,scale=0.5] at (8,0) {\textcolor{red}{$ x(3n+1) $}-th column};
				\draw[dash pattern=on 0pt off 4pt, line cap=round, thick] (9, -0.2) -- (10, -0.2);

				\fill[blue] (2, 0) circle (1.5pt); 
				\treeOne{2.8}{0}                     
				\treeTwo{5}{0}                     
				\treeN{8}{0}{row1}                       
				
				\draw[dash pattern=on 0pt off 4pt, line cap=round, thick] (6, 0.2) -- (7, 0.2);
				\draw[dash pattern=on 0pt off 4pt, line cap=round, thick] (9, 0.2) -- (10, 0.2);

				\fill[blue] (2+2, 2) circle (1.5pt); 
				\treeOne{2.8+2}{2}                     
				\treeTwo{5+2}{2}                     
				\treeN{8+2}{2}{row2}                       
				
				\draw[dash pattern=on 0pt off 4pt, line cap=round, thick] (6+2, 2) -- (7+2,2);
				\draw[dash pattern=on 0pt off 4pt, line cap=round, thick] (8+2+0.5, 2) -- (10+1.5, 2);

				\fill[blue] (2+4, 4) circle (1.5pt); 
				\treeOne{2.8+4}{4}                    
				\treeTwo{5+4}{4}                     
				\treeN{8+4}{4}{row3}                       
				
				\draw[dash pattern=on 0pt off 4pt, line cap=round, thick] (6+4, 4) -- (7+4,4);
				\draw[dash pattern=on 0pt off 4pt, line cap=round, thick] (8.5+4, 4) -- (10+3.6, 4);

				\fill[red] (8+2,2) circle (2pt);
				\node[scale=0.4, transform shape] at (12.8, 2) {\textcolor{red}{$ x(3\pi(n,t(0))+2) $}-th row};
				
				\fill[red] (Rrow2) circle (2pt);
				\node[right=0.6cm, align=right, scale=0.4, transform shape] at (Rrow2) {\textcolor{red}{$ x(3\pi(n,t(0),t(1))+2) $}-th No.\\of the domain};
				
				\fill[red] (R2row2) circle (2pt);
				\node[right=1.05cm, align=right, scale=0.4, transform shape] at (R2row2) {\textcolor{red}{$ x(3\pi(n,t^\ast)+2) $}-th No.\\of the domain};
				
				\fill[red] (treeNRRrow2) circle (2pt);
				\node[right=1.27cm, scale=0.4, transform shape] at (treeNRLrow2) {\textcolor{red}{$ x(3\pi(n,t)+2) $}-th No.};
			\end{tikzpicture}
			\caption{Definition of $\Phi(x)$ at $ A_{x(3n)} $ for a single $ t $, where $ t^\ast = t\uphar (|t|-1) $.}
			\label{fig:tildefinalpha}
		\end{figure}

		\begin{claim}{\label{claim:positive}}
			For any $ x\in \infsub{\bbn} $, we have $ \Phi(x)\in \finalpha{\alpha+} $.
		\end{claim}
		
		\begin{claimproof}{\ref{claim:positive}}
			We show that in fact $ \Phi(x) \in \finalpha{\alpha++}$. To prove this, we show that for any $ n\in\bbn  $ we have $ \Phi(x)(k_{x,n})\in \finalpha{\gamma(\alpha, \langle k_{x,n} \rangle)++} $ by Proposition \ref{prop:plusplusterminal}. We prove by induction that for any $ s\in\inffunc{\bbn} $ and $ i\in[1,|m(x,n,s)_{K_{x,n}}|] $,
			\[
			\Phi(x)(k_{x,n})[s\uphar i] = K_{x,n}[m(x,n,s)\uphar i],
			\]
			and hence $ s\uphar |m(x,n,s)_{K_{x,n}}|\in \terminalsequence(\Phi(x)(k_{x,n}))$.
			\begin{itemize}
				\item[(1)] For $ i=1 $, as $ K_{x,n} $ is a $ {}^{++} $-set, we have that for any $ t\in \inffunc{\bbn} $, $ |m(x,n,t)_{K_{x,n}}| \geq 1$, which implies
				\[
				\dom(\Phi(x)(k_{x,n})) = \{K_{x,n}[x(3\pi(n,l)+2)]: l\in\bbn\}.
				\]
				Consequently, the $ s(0) $-th number in this domain $ \Phi(x)(k_{x,n})[s\uphar 1] $ is just $ K_{x,n}[m(x,n,s)\uphar 1] $.
				\item[(2)]For $ i>1 $ with $ i\leq |m(x,n,s)_{K_{x,n}}| $, by the induction hypothesis we have 
				\[
				\Phi(x)(k_{x,n})[s\uphar (i-1)] = K_{x,n}[m(x,n,s)\uphar (i-1)].
				\]
				For convenience, set
				\[
				P = \Phi(x)(k_{x,n})\llb s\uphar (i-1)\rrb, \ Q = K_{x,n}\llb m(x,n,s)\uphar (i-1)\rrb.
				\]
				By the definition of $ \Phi(x) $, we have that
				\[
				P = \bigcup_{t\in\inffunc{\bbn}}Q[m(x,n,s\uphar(i-1)^\frown t)_Q].
				\]
				Since $ Q $ is non-empty and thus a $ {}^{++} $-set, for any $ t\in\inffunc{\bbn} $, we have $ |m(x,n,s\uphar(i-1)^\frown t)_Q|\geq 1 $. Consequently,
				\[
				\dom(P) = \{Q[x(3\pi(n,s\uphar(i-1) ^\frown l)+2)]: l\in\bbn\},
				\]
				and the $ s(i-1) $-th number in the domain is just $ x(3\pi(n,s\uphar i)+2) $-th number in $ \dom(Q) $. This completes the induction step.
			\end{itemize}
		\end{claimproof}
		
		We now show that $\Phi$ satisfies the two conditions of Definition \ref{def:dichotomous-coding}. The proof follows the same pattern as in Proposition \ref{prop:eddichotomous}, with the key difference that we now work in the tree $S_\alpha$ and use the function $\pi$ to index deeper levels.
		
		\begin{claim}{\label{claim:hide}}
			For every $H \in [\mathbb N]^\omega$, there exist $y \in [H]^\omega$ and $A \in \mathcal A$ such that $\Phi(y) \subseteq A$.
		\end{claim}
		
		\begin{claimproof}{\ref{claim:hide}}
			Since $\mathcal A$ is a $\finalpha{\alpha}$-MAD family and $\Phi(H) \in \finalpha{\alpha+}$, by maximality there exists $A \in \mathcal A$ such that $\Phi(H) \cap A \in \finalpha{\alpha+}$, and we have a subset $ B \subseteq \Phi(H) \cap A  $ with $ B\in \finalpha{\alpha++} $. 
			
			Now, we construct $y \in [H]^\omega$. For any $ i $ we let $ n_i {}^\frown t_i $ be the sequence satisfying $ \pi(n_i {}^\frown t_i) = i $. Note that $ |t_i| \geq 1$ and
			\[
			n_i\leq \pi(n_i {}^{\frown} (t_i \uphar 1))\leq \cdots \leq \pi(n_i {}^{\frown} (t_i \uphar j)) \leq \cdots \leq \pi(n_i {}^{\frown} t_i)=i.
			\]
			For convenience, set $ y(-1) = -1 $ and $ u_{-1} = \emptyset $. We recursively define  $y(3i)$, $y(3i+1)$, and $y(3i+2)$ for any $ i\in\bbn $ such that
			\[
			k_{y,n_i} \concatenate K_{y,n_i}[m(y,n_i , t_i \uphar (|t_i|-1 )) \uphar T_i]\in T(B)
			\]
			where 
			\[
			T_i = |(m(y,n_i , t_i \uphar (|t_i|-1 ))\concatenate \mathrm{Id})_{K_{y,n_i}}|,
			\]
			and $ \mathrm{Id} $ is the identity function on $ \bbn $.
			
			The construction proceeds as follows. For each $ i $,
			\begin{itemize}
				\item[(1)] Define $ m_i = \min\{m: k_{H,m}\in \dom(B) \wedge H(3m) > y(3i-1)\}$, and set $y(3i) = H(3m_i)$.
				\item[(2)] Set $ y(3i+1) = H(3m_i +1) $.
				\item[(3)] To define $ y(3i+2) $, by the induction hypothesis,
				\begin{itemize}
					\item[i.] If $ T_i\leq |t_i |-1 $, then set $ y(3i+2) = H(3m_i +2) $;
					\item[ii.] Otherwise, we denote
					\[
					H_l = H(3\pi(m_{n_i} ,\langle u_{\pi(n_i, t_i \uphar j)}: j <|t_i| \rangle, l )+2),
					\]
					then set
					\[
					\begin{aligned}
						u_i &= \min\{l:A_{y(3n_i)}[y(3n_i +1),m(y,n_i , t_i \uphar (|t_i|-1 )), H_l ]\\
						&\in T(B) \wedge H_l > y(3i+1)\},
					\end{aligned}
					\]
					and finally define $ y(3i+2) = H_{u_i} $. 
				\end{itemize}
			\end{itemize}
			
			Note that when defining $ y(2) $, the induction hypothesis is automatically satisfied, and we can define $ u_i $ because of Fact \ref{fact:plusplus}(3). This sufficies to complete the recursion. It is then routine to verify that $ \Phi(y)\subseteq B $.
		\end{claimproof}
		
		\begin{claim}{\label{claim:split}}
			For every $H \in [\mathbb N]^\omega$, there exist $y \in [H]^\omega$ and $A \in \mathcal A$ such that both $A \cap \Phi(y)$ and $\Phi(y) \setminus A$ belong to $\finalpha{\alpha+}$.
		\end{claim}
		
		\begin{claimproof}{\ref{claim:split}}
			As in the previous proof, we have $ A\in \idl{A}$ and $ B\in \finalpha{\alpha++} $ such that $ B\subseteq \Phi(H)\cap A $. We define $ y $ by hiding in and escaping from $ A $ alternately. We recursively define  $y(3i)$, $y(3i+1)$, and $y(3i+2)$ for any $ i\in\bbn $ such that
			\[
			\left\{\begin{aligned}
				&k_{y,n_i} \concatenate K_{y,n_i}[m(y,n_i , t_i^\ast) \!\uphar\! T_i]\in T(B), &\!\text{if } n_i \text{ is even},\\
				(A_{y(3 n_i)} \!\cap\! A)(&k_{y,n_i} \concatenate K_{y,n_i}[m(y,n_i , t_i^\ast) \!\uphar\! T_i]) \text{ is Fr\'echet small}, &\text{if } n_i \text{ is odd},
			\end{aligned}\right.
			\]
			where $t_i^\ast = t_i \!\uphar\! (|t_i|-1 ) $ and $ T_i $ is defined as in Claim \ref{claim:hide}.
			
			The construction proceeds as follows. 
			\begin{itemize}
				\item[(1)] If $ i $ is even,
				\begin{itemize}
					\item[i.] We define $ y(3i) $, $ y(3i+1) $ exactly as in Claim \ref{claim:hide}.
					\item[ii.] For $ y(3i+2) $, we consider $ n_i $. By the induction hypothesis, if $ n_i $ is even, we apply the construction as in Claim \ref{claim:hide}. If it is odd and $ T_i\leq |t_i |-1 $, define $ y(3i+2) = H(3m_i +2) $; otherwise, let
					\[
					\begin{aligned}
						y(3i+2) = \min\{u\in H:\ &(A_{y(3 n_i)} \!\cap\! A)\big(A_{y(3n_i)}[y(3n_i +1), m(y,n_i , t_i^\ast), u ]\big) \\
						&\text{ is Fr\'echet small } \wedge u>y(3i+1)\}.		
					\end{aligned}
					\]
				\end{itemize}
				\item[(2)] 	Now, if $ i $ is odd, 
				\begin{itemize}
					\item[i.] Define $m_i = \min\{m: H(3m) > y(3i-1)\}, \ y(3i) = H(3m_i)$.
					\item[ii.] Define
					\[
					\begin{aligned}
						y(3i+1) &= \min\{u\in H: (A_{y(3n_i)}\cap A) (A_{y(3n_i)}[u])\\
						&\in \finalpha{\gamma(\alpha, A_{y(3n_i)}[u])}\wedge u> y(3i)\}.
					\end{aligned}
					\]
					\item[iii.] For $ y(3i+2) $, the definition is exactly the same as the case when $ i $ is even.
				\end{itemize}
			\end{itemize}
			
			Note that $ 0 = \pi(0,0) $ and $ 1= \pi(1,0) $, and by Fact \ref{fact:plusplus}(3) together with the fact that $ A_{y(3n_i)}\cap A $ is Fr\'echet small for any $ i $, we can define $ y(3i+2) $. Thus the recursion can be carried out. Finally, we obtain that $ A $ splits $ \Phi(y) $.
		\end{claimproof}
		
		Since $\Phi$ satisfies both conditions of Definition \ref{def:dichotomous-coding}, we conclude that $\finalpha{\alpha}$ is a dichotomously coded ideal.
	\end{proof}
	
	We now combine the results of this section to obtain the main theorem on $\finalpha{\alpha}$-MAD families.
	
	\begin{thm}[$ \zfaxiom $]\label{theorem:ramseyimplynofin}
		Let $\Gamma$ be a good pointclass and let $1\leq \alpha < \omega_1$. If every set in $\Gamma$ has the Ramsey property, then $\Gamma$ contains no Dedekind-infinite $\finalpha{\alpha}$-MAD family.
	\end{thm}
	
	\begin{proof}
		For any countable $ \alpha$, the coding operator constructed in Proposition \ref{prop:findichotomous} is continuous. For $ \alpha = 1 $, it is proved in Theorem \ref{thm:mad-ramsey}. For $ \alpha >1 $, th continuity follows from the fact that for any $ p\in S_\alpha $,
		\[
		\begin{aligned}
			p\in \Phi(x) \text{ if and only if } &\exists n\in\bbn \  \exists t\in \bbn^{|p|-1} \  p = A_{x(3n)}[x(3n+1), m(x,n,t)],\\
			p\notin \Phi(x) \text{ if and only if } &\bigg[\forall n \suchthat A_{3n}[0]\leq p(0), \ p(0) \neq k_{x,n} \bigg] \ \vee \\
			&\bigg[ \exists n\in\bbn \  \big[ \ p(0) = k_{x,n} \  \wedge \  \forall t\in B(n,p)\\
			&\phantom{\bigg[ \exists n\in\bbn \  \big[ \ }p\neq A_{x(3n)}[x(3n+1), m(x,n,t)] \  \big] \bigg],
		\end{aligned}
		\]
		where
		\[
		B(n,p) = \left\{s\in \bbn^{|p|-1}: \forall i\in [1,|p|) \ 3\pi(n,s\uphar i)+2 \leq p(i) \right\},
		\]
		which is a finite set. Note that the second equivalence is true because
		\[
		A_{3n}[0]\leq A_{x(3n)}[0]< k_{x,n}
		\]
		and every evaluation of $ x(3\pi(n,t\uphar i)+2) $ is no less than $ 3\pi(n,t\uphar i)+2 $, for any $ i\geq1 $.
		
		Applying Proposition \ref{prop:ramseyimplynoedmad} with $\mathcal I = \finalpha{\alpha}$ together with the assumption that every set in $\Gamma$ has the Ramsey property yields the desired conclusion.
	\end{proof}
	
	Taking $\Gamma$ to be the pointclass of analytic sets, we obtain the following classical consequence.
	
	\begin{cor}[\cite{bakke2022maximal}]
		For every countable ordinal $\alpha < \omega_1$, there are no analytic infinite $\finalpha{\alpha}$-MAD families.
	\end{cor}

	\section{Questions and Comments}{\label{section:problem}}
	
	We conclude with discussions on some questions and directions for future research.

	\subsection{MAD families}
	
	The following well-known problem remains open.
	
	\begin{pblm}[{\cite[~Problem 4.1]{tornquist2018definability} \cite[~Question 7.4]{neeman2018happy}}]\label{pblm:admad}
		Does $\zfaxiom+\adaxiom$ imply that there are no infinite MAD families?
	\end{pblm}
	
	By Theorem \ref{theorem:introduction}, a positive solution for the following long-standing problem will solve Problem \ref{pblm:admad}:
	
	\begin{pblm}[{\cite[~Remark 8.12]{larson2023extensions}}]\label{pblm:adramsey}
		Does $\zfaxiom+\adaxiom$ \textnormal{(+ $ \dcaxiom_\bbr $)} imply that all sets of reals are Ramsey?
	\end{pblm}
	
	
	Since the method of tilde operators only works for Dedekind-infinite MAD families, the following question naturally arises.
	
	\begin{qst}
		Is $\zfaxiom$ + ``all sets are Ramsey'' + ``there is an infinite Dedekind-finite MAD family'' consistent?
	\end{qst}
	
	One can also ask whether a specific regularity property holding for a pointclass $\Gamma$ implies the non-existence of $\Gamma$-definable MAD families. A local test is the existence of definable MAD families in the generic extensions of their corresponding forcing notions. Note that it is a classical result of Miller that in G\"odel's universe $ L $ there is already a $ \Pi^1_1 $ (lightface) MAD family \cite{miller1989infinite}. For this test, generically indestructible MAD families serve as classical witnesses of counterexamples, as explained in \cite{brendle2013mad,schrittesser2019definable}: under the \emph{Continuum Hypothesis} $ \chaxiom $, well-known transfinite constructions produce indestructible MAD families with respect to Cohen forcing, Random forcing, etc., see \cite[Chapter~VIII]{kunen1980set}; working in $ L $, we have a $ \Delta^1_2 $ well-ordering of reals, along which we can recursively construct $ \Sigma^1_2 $ indestructible MAD families. Such families remain maximal and $ \Sigma^1_2 $ in Cohen, Random, Miller and Sacks extensions of $ L $, and furthermore impliy the existence of $ \Pi^1_1 $ MAD families in these extensions, as T\"ornquist showed that the existence of $ \Sigma^1_2 $ MAD families implies the existence of $ \Pi^1_1 $ MAD families \cite{tornquist2009sigma1}. For Laver forcing, while there are no Laver indestructible MAD families since it adds a dominating real \cite{brendle2005forcing}, Schrittesser--T\"ornquist recently proved that there is a $ \Pi^1_1 $ MAD family in the Laver extension of $ L $, where all $ \Sigma^1_2 $ sets are Laver measurable \cite{tornquist2025happy}. On the other hand, there are no $ \Pi^1_1 $ MAD families in the Mathias extension \cite{schrittesser2019ramsey}, where Mathias forcing also adds a dominating real. The raises the following questions:
	
	\begin{qst}
		\leavevmode
		\begin{enumerate}[label=(\arabic*), font=\upshape]
			\item Is there a good characterization of a $ \sigma $-ideal $ \idl{I} $ such that under $ \dcaxiom_\bbr $ \textnormal{($ \ccaxiom_\bbr $)} ``all sets are $ \idl{I} $-regular'' implies that there are no infinite MAD families? The notion of $ \idl{I} $-regularity is defined in Definition \ref{defn:i_regular}.
			\item Locally, is there a good characterization of a proper forcing $ \mathbb{P} $ such that there are no $ \Pi^1_1 $ MAD families in the $ \mathbb{P} $-generic extension of $ L $, or there are no $ \mathbb{P} $-indestructible MAD families?
		\end{enumerate}
	\end{qst}
	
	Reflecting on Mathias forcing and our results, especially on the construction of coding operators, a natural candidate for the characterization seems to be the fast growth of generic reals, such as adding dominating reals. While Schrittesser-T\"ornquist's result suggests that destroying MAD families is more than adding dominating reals \cite{tornquist2025happy}, currently we don't even know if we have to add dominating reals to destroy a MAD family. This amounts to asking the following well-known question.
	
	
	\begin{pblm}
		Let $ \mathfrak{a} $ be the \emph{almost disjointness number} and $ \mathfrak{d} $ be the \emph{dominating number}.
		
		\begin{itemize}[leftmargin=2.5cm, labelwidth=2cm, labelsep=0.2cm, align=right]
			\item[\textnormal{(\cite{guzman2017ppoints,brendle2025combinatorial})}] For every MAD family, is there a proper forcing that destroys it and does not add unbounded, dominating or unsplit reals?
			\item[\textnormal{(Roitman)}] Does $ \mathfrak{d} = \omega_1 $ imply $ \mathfrak{a} = \omega_1 $?
		\end{itemize}
	\end{pblm}
	

	\subsection{Definable Idealized MAD Families}
	
	For MAD families, $ \bfpi^1_1 $ is the optimal possibility since there are no analytic MAD families. For $ \idl{I} $-MAD families, our results show that they can not be analytic for $\mathcal{ED}$, $\mathcal{ED}_{\mathrm{fin}}$ and $\finalpha{\alpha}$. However, Farkas--Soukup mentioned the following fact:
	
	\begin{fact}[{\cite{farkas2009more}}]
		For every infinite almost disjoint family $\idl{A}\subseteq \infsub{\bbn} $ in a pointclass $ \Gamma $, there exists a proper ideal $ \idl{I} \in \forall^{\idl{N}}\Gamma$ such that $ \idl{A} $ is $ \idl{I} $-MAD. 
	\end{fact}
	
	\begin{proof}
		For any such $ \idl{A} $, define
		\[
		\idl{I} =\{x\in \prst(\bbn): \forall A\in\idl{A} \ x\cap A \text{ is finite} \}.
		\]
		Then $ \idl{I} $ is a proper ideal containing all finite sets, and $ \idl{A}\subseteq \idl{I}^+ $ is an $ \idl{I} $-MAD family. Since $ \idl{A}\in\Gamma $, we have $ \idl{I} \in  \forall^{\idl{N}}\Gamma$.
	\end{proof}
	
	As there is a classical example of a closed almost disjoint family of size continuum, we have that there is a $ \bfpi^1_1 $ ideal $ \idl{I} $ with an uncountable closed $ \idl{I} $-MAD family. Also, for the density zero ideal $ \idl{Z}_0 $, there is a countably infinite $ \idl{Z}_0 $-MAD family \cite[Theorem~2.2]{farkas2009more}. In their paper, Farkas--Soukup also adapted the construction of Cohen indestructible MAD families under $ \chaxiom $ to idealized ones, showing the existence of Cohen indestructible $ \idl{I} $-MAD families for analytic $ P $-ideals under $ \chaxiom $ \cite[Theorem~3.1]{farkas2009more}. Thus similarly, for any analytic $ P $-ideal $ \idl{I}\in \Pi^0_3(a) $, there is a $ \Sigma^1_2(a) $ $ \idl{I} $-MAD family in the Cohen extension of $ L[a] $.
	
	This raises the following questions.
	
	\begin{qst}
		\leavevmode
		\begin{enumerate}[label=(\arabic*), font=\upshape]
			\item What is the minimal possible complexity of an infinite (uncountable) $\mathcal I$-MAD family for a given analytic ideal $\mathcal I$?
			\item Is there a good characterization of an analytic ideal $ \idl{I} $ such that there is no infinite (uncountable) analytic $ \idl{I} $-MAD family?
		\end{enumerate}
	\end{qst}
	
	\begin{qst}
		\leavevmode
		\begin{enumerate}[label=(\arabic*), font=\upshape]
			\item Is $ \idl{I} $ dichotomously coded by a Borel coding operator for any $ F_\sigma $ ideal $ \idl{I} $?
			\item Is there an uncountable analytic $ \idl{Z}_0 $-MAD family?
		\end{enumerate}
	\end{qst}

	\subsection{MID Families}
	
	For MID families, corresponding results have also been obtained. By adapting a combinatorial machinery from Eisworth and Shelah to the classical recursion framework, Brendle--Fischer-Khomskii proved that there is a $ \bfsgm^1_2 $ Sacks indestructible MID family which is preserved in countable supported iterations of Sacks forcing, and they also showed that if there is a $ \bfsgm^1_2 $ MID family, then there is a $ \bfpi^1_1 $ MID family \cite{brendle2019definable}. It seems that in a classical setting, regularities cannot preserve MID families while destroying MAD families.
	
	Globally, our results show that classical regularities either only destroy MID families (the Baire property and the Lebesgue measurability), or destroy MAD and MID families at the same time (the Ramsey property). This phenomenon is also reflected in local generic extensions. Among the standard forcing notions in the literature, Sacks forcing is the only one for which the iterated models satisfy $ \mathfrak{i}=\omega_1 $, where $ \mathfrak{i} $ is the \emph{independence number}, while for $ \mathfrak{a} $ this is the case in Cohen, Random, Sacks and Miller extensions \cite[Table 4]{blass2009combinatorial}. As pointed out in \cite[Section~6]{cruz2023partition}, in their iterated models, these standard Borel forcing notions either increase the \emph{covering number} of the \emph{ideal of meager sets} $ \mathbf{cov}(\idl{B}) $ to the size of continuum (finite support iterations of Cohen, Random and Hechler) while in $ \zfcaxiom $ we have $\mathbf{cov}(\idl{B})\leq \mathfrak{i} $, or make the \emph{bounding number} $\mathfrak{b} =  \mathfrak{a} $ (countable support iterations of Sacks, Laver, Mathias and Miller), while $ \mathfrak{b}\leq \mathfrak{i} $ in $ \zfcaxiom $.
	
	Thus we have the following question.
	
	\begin{qst}
		Is there a $ \sigma $-ideal $ \idl{I} $ such that
		\begin{enumerate}[label=(\arabic*), font=\upshape]
			\item Under $ \dcaxiom_{\bbr} $ \textnormal{($ \ccaxiom_{\bbr} $)}, ``all sets are $ \idl{I} $-regular'' implies that there are no MAD families,
			\item But there is a $ \Pi^1_1 $ MID family in the $ \bbp_\idl{I} $-extension of $ L $?
		\end{enumerate}
	\end{qst}
	
	Locally, the following is a long-standing fundamental problem.
	
	\begin{pblm}[Vaughan]
		Is it consistent that $ \mathfrak{i}<\mathfrak{a} $?
	\end{pblm}
	
	We refer to \cite[Section~6]{cruz2023partition} for a comprehensive review on the obstacles and potential solutions to the Vaughan's problem.

	\subsection{Unifying Notions}
	
	Although we have had various examples of regularities that destroy pathologies, there are also some exceptions. Maximal eventually different (MED) families, as an analog to MAD families in the Baire space, have Borel and even compact instances \cite{horowitz2024eventually,schrittesser2018compactness}; Maximal cofinitary groups, another kind of combinatorial object related to MED families, turn out to have Borel and $F_\sigma$ instances \cite{horowitz2025borel,mejak2022cofinitary}. While all of these pathological sets can be realized by maximal $G$-independent (discrete) sets for some hypergraph $G$ \cite{schrittesser2020discrete,mejak2023definability}, they are in different levels of definability.
	
	We would like to pose a general question in a less formal way. Before that some definitions need to be introduced.
	
	\begin{defn}[{\cite[Definition 2.2.1]{khomskii2012regularity}}]\label{defn:i_regular}
		Let $\idl{I}$ be a $\sigma$-ideal on $\bairesp$, and $\bbp_\idl{I}$ be the poset of Borel $\idl{I}$-positive sets with the order given by inclusion. A subset $S\subseteq \bairesp$ is called $\idl{I}$-\textbf{regular} if and only if for any $B\in\bbp_\idl{I}$ there exists $A\in \bbp_\idl{I}$, such that $A\leq B$ and either $A\subseteq S$ or $A\cap S=\emptyset$. Note that we do not assume the properness of $\bbp_\idl{I}$.
	\end{defn}
	
	\begin{defn}[{\cite{schrittesser2020discrete}}]
		A \textbf{hypergraph} $G=(V,H)$ consists of a vertex set $V$ and a set of hyperedges $H$. It is called simple if it has no loops. It is called $<\omega$-uniform if every hyperedge is finite. We assume that a hypergraph is always simple and $<\omega$-uniform.
		
		A set $I\subseteq V$ is called \textbf{independent} (\textbf{discrete}) if no points in $I$ constitute a hyperedge, and is called maximal if it is maximal under subsets.
	\end{defn}

	Let $\choice$ be a sufficiently weak choice principle, such as $\dcaxiom_{\bbr} $, $\ccaxiom_{\bbr}$ or even weaker.
	The question is to find equivalent conditions on the $\sigma$-ideal $\idl{I}$ and hypergraph $G$ of reals such that the following postulate is true:
	\begin{postulate}[$\zfaxiom+\choice$]
		\leavevmode
		\begin{enumerate}[label=(\arabic*), font=\upshape]
			\item All analytic sets of reals are $\idl{I}$-regular.
			\item There are no analytic maximal infinite $G$-independent sets.
			\item $\adaxiom$ implies ``all sets of reals are $\idl{I}$-regular".
			\item ``All sets of reals are $\idl{I}$-regular" implies ``there are no maximal infinite $G$-independent sets".
		\end{enumerate}
	\end{postulate} 
	
	It is known that if $\bbp_\idl{I}$ is proper, then all analytic sets are $\idl{I}$-regular \cite[Proposition 2.2.3]{khomskii2012regularity}. In \cite{ikegami2022determinacy}, Ikegami proved that $\zfaxiom+\dcaxiom+\adaxiom_\bbr$ implies all sets are $\idl{I}$-regular for $\bbp_\idl{I}$ proper. For maximal independent sets, there are mostly concrete examples and not too much is known for general cases. It is shown by Z. Vidny\'anszky that for a type of analytic hypergraphs, $V=L$ implies the existence of $\bfpi^1_1$ maximal independent sets, see \cite{vidnyanszky2014transfinite} and \cite[~Theorem 1.5]{schrittesser2020discrete}. It is shown in \cite[~Theorem 1.2]{schrittesser2019definable} that for any $\Sigma^1_1$ binary relation $R$, there is a $\Delta^1_2$ maximal $R$-discrete set in the Sacks and Miller extensions of $L$. Horowitz--Shelah constructed a Borel graph $G$ such that $\zfaxiom+\dcaxiom+$ ``there are no maximal $G$-independent sets'' is equiconsistent with $\zfcaxiom+$ ``there is an inaccessible cardinal" \cite[~Theorem 2]{horowitz2019non}.

	\section*{Acknowledgments}
	
	The authors are grateful to Renling Jin, David Schrittesser, and Hang Zhang for many valuable comments and discussions that helped improve this paper.

	
	\bibliographystyle{plain}
	\bibliography{mybibtex}
\end{document}